\documentclass[12pt]{article}

\usepackage[english]{babel}
\usepackage[utf8]{inputenc}
\usepackage[autostyle, english = american]{csquotes}
\MakeOuterQuote{"}
\usepackage{listings}
\usepackage{keyval}
\usepackage{xcolor}
\usepackage{textcomp}
\usepackage{fullpage}
\usepackage{comment}
\usepackage{setspace}
\usepackage{graphicx}
\usepackage{hyperref}
\usepackage{verbatim}
\usepackage{multicol}
\usepackage{float}
\usepackage{framed}
\usepackage{parskip}
\usepackage{pifont}
\usepackage{cancel}
\usepackage{algorithm2e} 
\usepackage{tcolorbox} 
\usepackage{mdframed} 
\usepackage{lipsum}
\usepackage{xcolor}
\usepackage{fullpage}

\graphicspath{ {Figures/} }

\usepackage{amsmath}
\usepackage{amssymb}
\usepackage{amsfonts}
\usepackage{amsthm}
\usepackage{marvosym}
\usepackage{wasysym}
\usepackage{slashbox}

\usepackage{tikz}
\usepackage{tkz-euclide}
\usetikzlibrary{positioning}
\usetikzlibrary{arrows}
\usepackage[toc,page]{appendix}  

\newtheorem{theorem}{Theorem}[section]
\newtheorem{thm}{Theorem}[section]

\newtheorem{lemma}{Lemma}[section]

\newtheorem{alg}[thm]{Algorithm}

\newcommand{\R}{\mathbb R}

\newcommand{\lip}[2]{\left({}#1,#2\right){}}

\newcommand{\curly}[1]{\left\{{}#1\right\}}
\newcommand{\eps}{\varepsilon}

\newcommand{\inv}{^{-1}}

\newcommand{\norm}[1]{\left|\left|{}#1\right|\right|}

\newcommand{\trinorm}[1]{{\left\vert\kern-0.2ex\left\vert\kern-0.2ex\left\vert #1 
    \right\vert\kern-0.2ex\right\vert\kern-0.2ex\right\vert}}
\newcommand{\abs}[1]{\left|{}#1\right|} 
\newcommand{\pare}[1]{\left({}#1\right)}
\newcommand{\brak}[1]{\left[{}#1\right]}

\newcommand{\blue}[1]{\textcolor{black}{#1}}

\newcommand{\bmat}[1]{\begin{bmatrix} #1 \end{bmatrix}}

\newcommand{\divergence}{\nabla \cdot}

\newcommand{\omint}{\int_{\Omega}}

\newcommand{\chalf}{^{n+\frac{1}{2}}}
\newcommand{\stokes}[1]{I_h^{St}\left({}#1\right)}

\usepackage{bm}

\newtheorem{remark}{Remark}[section]

\newcommand{\be}{{e}}

\newcommand{\bw}{{w}}

\newcommand{\bD}{{D}}

\newcommand{\bX}{{X}}
\newcommand{\bY}{{Y}}

\newcommand{\bfphi}{\ensuremath{{\phi}}}
\newcommand{\bfeta}{\ensuremath{{\eta}}}
\newcommand{\bfchi}{\ensuremath{{\chi}}}

\newcommand{\uh}{\ensuremath{{u}_h}}

\newcommand{\delt}{\ensuremath{\Delta t}}

\newcommand{\Grad}{\ensuremath{\nabla}}
\newcommand{\bea}{\begin{eqnarray}}
\newcommand{\eea}{\end{eqnarray}}
\newcommand{\beas}{\begin{eqnarray*}}
\newcommand{\eeas}{\end{eqnarray*}}

\numberwithin{equation}{section}
  
\begin{document}

\title{Improved long time accuracy for projection methods for Navier-Stokes equations using EMAC formulation}

\author{
Sean Ingimarson
 \thanks{School of Mathematical Sciences and Statistics, Clemson University, Clemson, SC 29634 (singima@clemson.edu); partially supported by NSF grant DMS 2011490}
 \and 
Monika Neda
\thanks{Department of Mathematical Sciences, University of Nevada Las Vegas, Las Vegas, NV 89154 (Monika.Neda@unlv.edu)}
\and
 Leo Rebholz
 \thanks{School of Mathematical Sciences and Statistics, Clemson University, Clemson, SC 29634 (rebholz@clemson.edu); partially supported by NSF grant DMS 2011490}
\and  
 Jorge Reyes
 \thanks{Department of Mathematical Sciences, University of Nevada Las Vegas, Las Vegas, NV 89154 (Jorge.Reyes@unlv.edu)}
 \and
 An Vu
 \thanks{Department of Mathematics, University of Houston, Houston, TX 77004 (atvu15@uh.edu)}
 }

\date{}
 
\maketitle

\abstract{
We consider a pressure correction temporal discretization for the incompressible Navier-Stokes equations in EMAC form.  We prove stability and error estimates for the case of mixed finite element spatial discretization, and in particular that the Gronwall constant's exponential dependence on the Reynolds number is removed (for sufficiently smooth true solutions) or at least significantly reduced compared to the commonly used skew-symmetric formulation.  We also show the method preserves momentum and angular momentum, and while it does not preserve energy it does admit an energy inequality.  Several numerical tests show the advantages EMAC can have over other commonly used formulations of the nonlinearity.  Additionally, we discuss extensions of the results to the usual Crank-Nicolson temporal discretization.
}

\section{Introduction}

It is widely accepted that the Navier-Stokes equations (NSE) determine the evolution of incompressible, viscous, Newtonian fluid flow.  These equations are given by
\begin{align}
u_t+u \cdot \nabla u+\nabla p-\nu \Delta u &=f,\label{eq:Strong_NSE1}\\
\divergence u&=0,\label{eq:Strong_NSE2}
\end{align}
where $u$ and $p$ represent velocity and pressure respectively, $f$ is an external force, and $\nu$ represents kinematic viscosity, which is inversely proportional to the Reynolds number $Re$.  Appropriate boundary and initial conditions are needed to close the system.

While the NSE are built from conservation of linear momentum and mass conservation, they are also well-known to conserve energy, angular momentum, enstrophy in 2D, helicity in 3D, \textcolor{black}{among other important physical quantities} \cite{GS98}.  By `conserve' we refer to the case of no viscous or external forces, but if these forces are present than an exact balance can be derived where the nonlinearity plays no role.  In addition to being conserved, these quantities are believed to play a critical role in flow structure development, the energy cascade and energy dissipation, and the microscale \cite{F95,GS98,R07}.  However, in most NSE simulations, very few or none of these quantities are exactly conserved \cite{CHOR17,CHOR19}.  Often, energy is at least bounded, as this is required for numerical stability.  However, in most finite element computations mass is only weakly conserved \cite{JLMNR17}, and this in turn breaks the conservation of momentum, angular momentum and other important physical quantities \cite{CHOR17}.  One solution to this problem is to use strongly divergence-free discretizations, such as Scott-Vogelius finite elements, however these elements can require mesh restrictions and higher degree polynomials, especially in the case of quadrilateral elements.

Another approach is to change the form of the nonlinearity to the EMAC (Energy, Momentum, and Angular momentum Conserving) form proposed in \cite{CHOR17}, where the identity
\begin{align*}
u \cdot \nabla u+\nabla p=2D(u)u+(\divergence u)u+\nabla P,
\end{align*}
was derived, with $P=p-\frac{1}{2}|u|^2$.  There it was shown that the NSE with this nonlinear formulation used in \eqref{eq:Strong_NSE1}-\eqref{eq:Strong_NSE2} and discretized with standard elements such as Taylor-Hood or the mini element, conserves energy, momentum and angular momentum, as well as particular definitions of 2D enstrophy and \textcolor{black}{3D} helicity.  This is in contrast to the more commonly used rotational, skew-symmetric, convective and conservative forms, none of which conserve all of energy, momentum and angular momentum \cite{CHOR17}.  

Since the original EMAC paper \cite{CHOR17} in 2017, EMAC has garnered a considerable amount of attention in the CFD community.  It has been used in problems involving vortex-induced vibration \cite{PCLRH18}, turbulent flow simulation \cite{LHOCR19}, noise radiated by an open cavity \cite{MSLGD19}, high Reynolds number vortex dynamics \cite{SL18}, and more \cite{SP18,SP18b,OCAM19,LPH19,Belding22}.  These numerical results have all been quite favorable, but there is still much to be done for its analytical study.  What is proven so far is results for conservation properties \cite{CHOR19}, stability and convergence \cite{CHOR19}, efficient algorithms and linearization development \cite{CHOR19}, and a longer time accuracy result that shows the Gronwall exponent from EMAC is independent of the viscosity \cite{OR20}.

The purpose of this paper is to extend the study of EMAC to the case of a projection method temporal discretizations together with finite element spatial discretization.  Projection methods were originally developed by Temam \cite{T69} and Chorin \cite{Cho68}, and work using a Hodge type decomposition idea to split the NSE into two steps: the first solves the momentum equation without a divergence-free constraint, and the second projects the step 1 solution into the divergence free space.  There have been many improvements made to projection methods over the years\footnote{The folklore, as told to author LR by a former Chorin student, is that for many years no Chorin student was allowed to graduate without improving on projection methods.}, but they all are still based on the fundamental decomposition \textcolor{black}{/} splitting from the original development.  Analysis of projection methods is rather different and more complex than for standard coupled schemes, see e.g. \cite{GMS06,P97,shen1992error}, and herein we will extend the study of EMAC discretizations using to projection methods.

\textcolor{black}{This paper is organized as follows. In section 2, we provide mathematical notation and preliminary information for the analysis. In section 3 we introduce the projection method algorithm and show the conservation properties of it. Stability and error analysis are presented in section 4. Section 5 further extends out work to coupled schemes for both EMAC and SKEW. Numerical tests can be found in section 6 followed by concluding remarks in the last section 7. }


\section{Notation and Preliminaries}
\label{sec:Notation}

We present in this section the necessary notation and mathematical preliminaries for a smooth analysis to follow.
We assume a convex polygonal (or smooth boundary) domain $\Omega \subseteq \R^d$ where $d=2,3$.  The $L^2(\Omega)$ inner product is denoted as $\lip{\cdot}{\cdot}$ and the $L^2(\Omega)$ norm with $\norm{\cdot}$.  Other norms will be clearly labeled with subscripts.

The natural velocity and pressure spaces for NSE are respectively denoted 
\begin{align*}
X=\curly{v \in H^1(\Omega), \: v|_{\partial \Omega}=0},
\qquad
Q=\curly{q \in L^2(\Omega), \: \omint q\: dx=0}.
\end{align*}
The dual of $X$ norm is defined as
\begin{align*}
X'=H\inv(\Omega), \qquad \norm{f}_{X'}=\sup_{v \in X'}\frac{\lip{f}{v}}{\norm{v}_X}.
\end{align*}
The divergence-free subspace of $X$ will be denoted by
\begin{align*}
V=\curly{v \in X, \: \lip{\divergence v}{q}=0, \: \forall q \in Q}.
\end{align*}
For projection methods, the space 
\[
Y = \{ v \in L^2(\Omega),\ \nabla \cdot v \in L^2(\Omega), \ v \cdot \hat n |_{\partial \Omega}=0 \}\textcolor{black}{,}
\]
is also utilized.

We use the following notation for nonlinear terms.  Let $c:X\times X\times X \rightarrow \mathbb{R}$ be defined by
\[
c(u,v,w):= (2D(u)v,w) + ((\nabla \cdot u)v,w).
\]
This formulation corresponds to the EMAC nonlinearity.  For skew-symmetric (SKEW), we define $b^*:\textcolor{black}{X}\times X\times X \rightarrow \mathbb{R}$ by
\[
b^*(u,v,w):= (u\cdot\nabla v,w) + \frac12 ((\nabla \cdot u)v,w).
\]
For $c$ and $b^*$ we have the bound from H\"older that
\begin{align*}
    c(u,v,w) & \le 2 \| D(u) \| \| v \|_{L^6} \| w \|_{L^3} + \| \nabla u \| \| v \|_{L^6} \| w \|_{L^3}, \\
    b^*(u,v,w)& \le  \| u \|_{L^6} \| \nabla v \| \| w \|_{L^3} + \frac12 \| \nabla u \| \| v \|_{L^6} \| w \|_{L^3},
\end{align*}
and thus with standard Sobolev inequalities we obtain the same bound for both $c$ and $b^*$ formulations,
\begin{align}
 b^*(u,v,w) & \le C \| \nabla u \| \| \nabla v \| \| w \|^{1/2} \| \nabla w \|^{1/2}, \mbox{ and } \nonumber \\
 c(u,v,w) & \le C \| \nabla u \| \| \nabla v \| \| w \|^{1/2} \| \nabla w \|^{1/2}. \label{cbound}
\end{align}

\subsection{Discretization Preliminaries}

We assume a regular conforming triangulation $\tau_h$ of $\Omega$, where $h$ is the global mesh size.  We further consider $X_h \subset X$ and $Q_h \subset Q$ as finite element velocity and pressure spaces.  We further assume that
$X_h=X\cap P_k(\tau_h)$ and $Q_h=Q\cap P_{k-1}(\tau_h)\cap C^0(\Omega)$.  On these spaces, the following condition holds \cite{GP08}: \textcolor{black}{there exists $\beta>0$ independent of $h$ such that}
\begin{equation}
\beta \| \nabla q \| \le \sup_{0\ne v \in X_h} \frac{(\nabla q,v)}{\| v \| } , \ \forall q \in Q_h. \label{LBB2}
\end{equation}
Note that \eqref{LBB2} is a stronger assumption than the usual inf-sup condition, in fact it implies it \cite{GP08}. \textcolor{black}{Besides Taylor-Hood we note that there are other element choices such as the mini element which would also satisfy \eqref{LBB2}.}

We also define the discretely divergence-free subspace of $X_h$ as
\begin{align*}
V_h=\curly{v_h \in X_h, \: \lip{\divergence v_h}{q_h}=0, \: \forall q_h \in Q_h},
\end{align*}
and the space $Y_h = Y \cap P_k(\tau_h)$.  Note that $X_h\subset Y_h$, and thus \eqref{LBB2} holds with $X_h$ replaced by $Y_h$.

For functions $v(x,t)$ and $1\leq p \leq \infty$, we use the notation
\begin{align*}
&\trinorm{v}_{\infty,k}=\max_{0\leq n \leq M}\norm{v^n}_k, \qquad \qquad\trinorm{ v }_{p,k}=\pare{\Delta t \sum_{n=0}^M\norm{v^n}_k^p }^{\frac{1}{p}},\\ 
& \qquad \qquad \qquad\trinorm{ v_{1/2} }_{p,k}=\pare{\sum_{n=0}^M\norm{v^{n-1/2}}_k^p\Delta t}^{\frac{1}{p}} \textcolor{black}{.}
\end{align*}

Denote $I_h^{St}$ as the discrete Stokes projection operator \cite{SL17}, which is defined by: Given $\phi \in H^1(\Omega)$, find $\stokes{\phi}\in V_h$ satisfying
\begin{align}
\lip{\nabla \stokes{\phi}}{\nabla v_h}=\lip{\nabla \phi}{\nabla v_h}, \quad \forall v_h \in V_h.
\label{eq:stokes}
\end{align}
For $\phi\in V \cap H^{k+1}(\Omega)$, this operator is known to have the following optimal approximation properties:
\begin{align}
\begin{split}
\norm{\phi-\stokes{\phi}}\leq Ch^{k+1}\norm{\phi}_{k+1},\\
\norm{\nabla (\phi - \stokes{\phi} ) }\leq Ch^k\norm{\phi}_{k+1}.
\end{split}
\end{align}
Our analysis will also utilize the following bound proven in \cite{girault2015max}:
\begin{align}
\norm{\nabla I_h^{St}(\phi)}_{L^r}\leq C\norm{\nabla \phi}_{L^r}, \quad r \in [2,\infty].
\label{eq:Stokes_approx}
\end{align}

\section{Projection Methods using EMAC}
\label{sec:Proj_Methods}

We now study projection methods for NSE with the EMAC nonlinearity.  For simplicity, we consider the backward Euler finite element scheme, \textcolor{black}{ but we note that the ideas can also be applied to higher order time discretization of projection methods \cite{GMS06, P08}.}

\begin{alg}[EMAC Projection Method: EMAC-BE-PROJ] \label{beleproj} \ \\
Let $f \in L^{\infty}(0,T;H^{-1}(\Omega))$, solenoidal initial condition $u^0 \in L^2(\Omega)$ {\color{black} satisfying no-slip boundary conditions}, $u_h^0=\widetilde{u}_h^0$ defined to be the $L^2$ projection of $u^0$ into $V_h$,  end time $T$, and number of time steps $M$ be given.  Set $\Delta t=T/M$, and for $n=0,1,2,...,M-1$,\\ \ \\
Step 1 EMAC: Find $\widetilde{u}_h^{n+1}\in X_h$ satisfying, for all $\bfchi_h\in X_h$,
\begin{align}
\frac{1}{\Delta t}(\tilde u_h^{n+1}- {u}_h^n,\bfchi_h) + c(\tilde u_h^{n+1},\tilde u_h^{n+1},\bfchi_h) + \nu(\nabla \tilde u_h^{n+1},\nabla \bfchi_h)
 =  (f (t^{n+1}),\bfchi_h ). \label{projemac1}
\end{align}
Step 2: Find $({u}_h^{n+1},{P}_h^{n+1})\in (Y_h,Q_h)$ satisfying, for all $(w_h,q_h)\in (Y_h,Q_h)$,
\begin{subequations}
\label{proj23}
\begin{align}
\frac{1}{\Delta t}({u}_h^{n+1},w_h) - (P_h^{n+1},\nabla \cdot w_h) & =  \frac{1}{\Delta t}(\widetilde{u}_h^{n+1},w_h), \label{proj2} \\
(\nabla \cdot {u_h}^{n+1},q_h) & =  0. \label{proj3}
\end{align}
\end{subequations}
\end{alg}

\begin{remark}
For the more commonly used skew-symmetric form projection method (SKEW-BE-PROJ), Step 1 EMAC would be replaced by Step 1 SKEW: Find $\widetilde{u}_h^{n+1}\in X_h$ satisfying, for all $\bfchi_h\in X_h$,
\begin{align}
\frac{1}{\Delta t}(\tilde u_h^{n+1}- {u}_h^n,\bfchi_h) + b^{*}(\tilde u_h^{n+1},\tilde u_h^{n+1},\bfchi_h) + \nu(\nabla \tilde u_h^{n+1},\nabla \bfchi_h)
 =  (f (t^{n+1}),\bfchi_h ). \label{projskew1}
\end{align}
Although step 2 of SKEW-BE-PROJ is the same as for EMAC-BE-PROJ, the interpretation of the pressure term is different.  In SKEW-BE-PROJ, it represents usual pressure, while for EMAC-BE-PROJ, it represents $p-\frac12|u|^2$.

One could also linearize via $b^{\ast}(u_h^{n},\tilde u_h^{n+1},\bfchi_h)$ or $b^{\ast}(\tilde u_h^{n},\tilde u_h^{n+1},\bfchi_h)$ for efficiency, {\color{black} although this could create additional error sources including an additional first order consistency term and can nonphysically alter  conservation of linear and angular momentum conservation in certain settings.  The extent to which these errors are significant is problem dependent. 

A fully explicit treatment of the nonlinear term would have a much more significant effect on the analysis, as it would create additional consistency error terms, a violation of energy and angular momentum conservation, and a time step restriction for stability and accuracy.  However, an interesting open question is how an explicit treatment using Scalar Auxilliary Variable technique (see e.g. \cite{LS20}) might be analyzed with respect to EMAC and conservation properties.}
\end{remark}

\subsection{Conservation properties of EMAC-BE-PROJ}
\label{sec:Proj_conservation}

It is well known that smooth solutions to the NSE conserve important quantities such as energy, momentum, and angular momentum, which are defined as
\begin{align*}
\text{Kinetic energy}\quad & E:= \frac{1}{2} \int_\Omega |u|^2\: dx;\\
\text{Linear momentum}\quad& M:=\int_\Omega u \: dx;\\
\text{Angular momentum}\quad& M_X:=\int_\Omega u\times x \: dx.
\end{align*}

We now consider these conservation laws for EMAC-BE-PROJ, and for comparison also SKEW-BE-PROJ.

\subsubsection{Energy inequality}

While neither SKEW-BE-PROJ nor EMAC-BE-PROJ conserve energy, they both admit an energy inequality.  This can be seen by choosing $\bfchi_h=\tilde u_h^{n+1}$ in \eqref{projskew1} for SKEW-BE-PROJ  and \eqref{projemac1} for EMAC-BE-PROJ, which both yield
\[
\frac{1}{2\Delta t} \left( \| \tilde u_h^{n+1} \|^2 - \| u_h^n \|^2 + \| \tilde u_h^{n+1} - u_h^n \|^2 \right) + \nu \| \nabla \tilde u_h^{n+1} \|^2 = (f(t^{n+1}),\tilde u_h^{n+1}),
\]
thanks to the polarization identity and $b^{\ast}(\tilde u_h^{n+1},\tilde u_h^{n+1},\tilde u_h^{n+1}) =c(\tilde u_h^{n+1},\tilde u_h^{n+1},\tilde u_h^{n+1})=0$.  Step 2 is an $L^2$ projection from $X_h\subset Y_h$ into the discretely divergence free subspace of $Y_h$, and thus $\| u_h^{n} \| \le \|\tilde u_h^{n} \|$.  Using this bound along with a standard treatment of the right hand side term, we obtain
\[
\frac{1}{\blue{2}\Delta t} \left( \| \tilde u_h^{n+1} \|^2 - \| \tilde u_h^n \|^2 + \| \tilde u_h^{n+1} - u_h^n \|^2 \right) + \nu \| \nabla \tilde u_h^{n+1} \|^2 \le  \nu^{-1} \| f(t^{n+1}) \|_{\blue{X'}}^2,
\]
and after summing over time steps,
\begin{equation}
\| \tilde u_h^{M} \|^2 + \nu \left( \Delta t \sum_{n=1}^M \| \nabla \tilde u_h^n \|^2 \right) +   \left(  \sum_{n=1}^M \| \tilde u_h^n - u_h^{n-1} \|^2 \right) \le \| \tilde u_h^0 \|^2 +  \nu^{-1}  \left( \Delta t \sum_{n=1}^M \|  f(t^{n+1}) \|_{\blue{X'}}^2 \right). \label{E1}
\end{equation}
Hence we find for both SKEW-BE-PROJ and EMAC-BE-PROJ an energy inequality instead of an equality.  This is due to dissipation from backward Euler time stepping and in the projection step, since backward Euler produces the third left hand side term in \eqref{E1}, while the projection step yielded the important bound  $\| u_h^{n} \| \le \|\tilde u_h^{n} \|$.   However, that inequality was possible due to the nonlinear terms in both SKEW-BE-PROJ and EMAC-BE-PROJ preserving energy.

\subsubsection{Momentum conservation}

To study momentum conservation properties, we consider a simplified setting where the data and solution vanishes in a strip along the boundary.  For example, one can consider the situation of an isolated spinning vortex, with the boundary sufficiently far away.  Of course, boundaries play an important role in balances of these quantities,  however considering this simplified setting will reveal how a scheme behaves away from the boundaries and in particular whether the nonlinear term preserves momentum or nonphysically contributes to it. Hence, for a given mesh $\tau_h$ of $\Omega = \overline{\Omega} \oplus \Omega_s$, where $\Omega_s$ is the part of the domain overlapping with the strip of elements along the boundary.  {\color{black} Define ${\psi_i}_h \in X_h$ to be the standard $i^{th}$ basis vector in all of the domain except in the final strip of the mesh along the boundary, where it decays to 0.}

From Step 2 of SKEW-BE-PROJ and EMAC-BE-PROJ, since ${\psi_i}_h \in X_h\subset Y_h$, choosing $w_h={\psi_i}_h$ yields
\[
(u_h^{n+1},{\psi_i}_h) - \Delta t (p_h^{n+1},\nabla \cdot {\psi_i}_h) = (\tilde u_h^{n+1},{\psi_i}_h),
\]
and since $\nabla \cdot \psi_i=0$,
\[
 \Delta t (p_h^{n+1},\nabla \cdot {\psi_i}_h)  =  \Delta t (p_h^{n+1},\nabla \cdot {\psi_i}_h)_{\Omega_s} = 0 \textcolor{black}{,}
 \]
 with the last quantity vanishing since we assume the solution vanishes in $\Omega_s$.  Hence
 \begin{equation}
 (u_h^{n+1},{\psi_i}_h) = (\tilde u_h^{n+1},{\psi_i}_h). \label{M1}
\end{equation}

Consider now choosing $\bfchi_h={\psi_i}_h$ in Step 1 of  EMAC-BE-PROJ.  Using \eqref{M1}, we obtain
\begin{align}
\frac{1}{\Delta t}(\tilde u_h^{n+1}- \tilde {u}_h^n,{\psi_i}_h) + c(\tilde u_h^{n+1},\tilde u_h^{n+1},{\psi_i}_h) + \nu(\nabla \tilde u_h^{n+1},\nabla {\psi_i}_h)
 =  ( f (t^{n+1}),{\psi_i}_h ).
\end{align}
Using that the solution and data vanish in $\Omega_s$, we can write
\begin{align}
\frac{1}{\Delta t}(\tilde u_h^{n+1}- \tilde {u}_h^n, \psi_i ) + c(\tilde u_h^{n+1},\tilde u_h^{n+1},{\psi_i})
 =  (f (t^{n+1}),{\psi_i} ),
\end{align}
noting the viscous term drops since $\nabla \psi_i$ vanishes.  The nonlinear term is also zero, as shown in \cite{CHOR17}, but which we show now for completeness.  Expanding $c$, we get
\begin{align*}
 c(\tilde u_h^{n+1},\tilde u_h^{n+1},{\psi_i}) & =  2(D(\tilde u_h^{n+1}) \tilde u_h^{n+1},{\psi_i}) +  ( (\nabla \cdot \tilde u_h^{n+1})\tilde u_h^{n+1},{\psi_i})  \\
& =  (\tilde u_h^{n+1}\cdot \nabla \tilde u_h^{n+1},{\psi_i}) +   ({\psi_i}\cdot \nabla \tilde u_h^{n+1},\tilde u_h^{n+1})  + ( (\nabla \cdot \tilde u_h^{n+1})\tilde u_h^{n+1},{\psi_i}) \\
& =  - (\tilde u_h^{n+1}\cdot \nabla {\psi_i},\tilde u_h^{n+1})  -  ( (\nabla \cdot \tilde u_h^{n+1})\tilde u_h^{n+1},{\psi_i})  \\
&  \ \ \ \ \  -\frac12   ( (\nabla \cdot {\psi_i}) \tilde u_h^{n+1},\tilde u_h^{n+1})  + ( (\nabla \cdot \tilde u_h^{n+1})\tilde u_h^{n+1},{\psi_i}),
\end{align*}
with the last step thanks to integrating by parts.  Notice the 2nd and 4th terms sum to zero, and the first and third are both 0 since $\psi_i$ is constant and therefore its derivatives are zero.  Thus $ c(\tilde u_h^{n+1},\tilde u_h^{n+1},{\psi_i}) =0$.  This leaves
\[
\tilde M_i^{n+1} - \tilde M_i^n = \Delta t f_i^{n+1},
\]
where $\tilde M_i^n := (\tilde {u}_h^n, \psi_i ).$ From \eqref{M1}, we also have that
\[
M_i^{n+1} - M_i^n = \Delta t f_i^{n+1},
\]
where $ M_i^n := ( {u}_h^n, \psi_i ).$ This establishes that for EMAC-BE-PROJ, the momentum balance is analogous to that of the continuous NSE and in particular that the nonlinear term does not contribute to the balance.

For SKEW-BE-PROJ, we are not able to obtain such a balance.  All terms except the nonlinear term are handled the same as for EMAC-BE-PROJ, and thus we consider
\begin{align*}
b^{\ast}(\tilde u_h^{n+1},\tilde u_h^{n+1},{\psi_i}_h) & = b^{\ast}(\tilde u_h^{n+1},\tilde u_h^{n+1}, \psi_i ) \\
 & = (\tilde u_h^{n+1}\cdot \nabla \tilde u_h^{n+1}, \psi_i )  + \frac12 ( (\nabla \cdot \tilde u_h^{n+1}) \tilde u_h^{n+1}, \psi_i ) \\
 & = - (\tilde u_h^{n+1}\cdot \nabla \psi_i , \tilde u_h^{n+1} )  - \frac12 ( (\nabla \cdot \tilde u_h^{n+1}) \tilde u_h^{n+1}, \psi_i )\\
 & =  - \frac12 ( (\nabla \cdot \tilde u_h^{n+1}) \tilde u_h^{n+1}, \psi_i ),
 \end{align*}
with the last steps thanks to integrating by parts and using that $\psi_i$ is divergence-free.  Since $\nabla \cdot \tilde u_h^{n+1}\ne 0$, this term is not expected to vanish.  Thus the momentum balance for SKEW-BE-PROJ is
\[
M_i^{n+1} - M_i^n = \tilde M_i^{n+1} - \tilde M_i^n = \Delta t f_i^{n+1} + \frac{\Delta t}{2}  ( (\nabla \cdot \tilde u_h^{n+1}) \tilde u_h^{n+1}, \psi_i ).
\]


Despite the analysis above, which would apply to a general setting although with additional terms, in the case of homogeneous boundary conditions we can show that both EMAC-BE-PROJ and SKEW-BE-PROJ conserve momentum since for any $n$ and $q_h\in Q_h$, 
\[
0 = (\nabla \cdot u_h^{n},q_h) = - (u_h^{n},\nabla q_h),
\]
from the conservation of mass constraint.  Now for $q_h = x_i - \int_{\Omega} x_i$ we obtain $0=(u_h^n,e_i)$.  Thus $M^n =0$ for all $n$, which implies momentum conservation with $M$ and using \eqref{M1} produces momentum conservation with $\tilde M$.

\subsubsection{Angular momentum conservation}

{\color{black} To determine the angular momentum balances, we proceed similar to the case above for momentum and make the same assumptions about the solution vanishing on a strip along the boundary. 
Define $\phi_i=x \times \psi_i$, and note that $\Delta \phi_i=0$ and $\nabla \cdot \phi_i=0.$
Here we use the test function $\bfchi_h={\phi_i}_h \in X_h$ where ${\phi_i}_h = \phi_i$ except it is 0 in the final strip of the mesh along the boundary.}

Similar arguments as in the case of momentum imply that from Step 2 with $w_h={\phi_i}_h$, we get
 \begin{equation}
 (u_h^{n+1},{\phi_i}_h) = (\tilde u_h^{n+1},{\phi_i}_h). \label{AM1}
\end{equation}

Consider now choosing $\bfchi_h={\phi_i}_h$ in Step 1 of  EMAC-BE-PROJ.  Using \eqref{AM1}, we obtain
\begin{align}
\frac{1}{\Delta t}(\tilde u_h^{n+1}-  {u}_h^n,
{\phi_i}_h) + c(\tilde u_h^{n+1},\tilde u_h^{n+1},{\phi_i}_h) + \nu(\nabla \tilde u_h^{n+1},
\nabla {\phi_i}_h)
 =  (f (t^{n+1}),{\phi_i}_h ).
\end{align}
Since the solution and data vanish in $\Omega_s$ and $\nabla \cdot \phi_i$, we can write
\begin{align}
\frac{1}{\Delta t}(\tilde u_h^{n+1}- \tilde {u}_h^n, \phi_i ) + c(\tilde u_h^{n+1},\tilde u_h^{n+1},{\phi_i})  + \nu(\nabla \tilde u_h^{n+1},\nabla {\phi_i})
 =  (f (t^{n+1}),{\phi_i} ).
\end{align}
The viscous term can be seen to vanish from a calculation, e.g. for $i=2$,
\[
(\nabla \tilde u_h^{n+1},\nabla {\phi_2}) = \left( \partial_z (\tilde u_h^{n+1})\textcolor{black}{_1} - \partial_x (\tilde u_h^{n+1})\textcolor{black}{_3,1} \right) = 0,
\]
and similar calculations can be made for $i=1,3$.

For the nonlinear term of EMAC-BE-PROJ, expanding $c$ just as in the momentum conservation section above and using that $\nabla \cdot \phi_i=0$, we get
\begin{align*}
 c(\tilde u_h^{n+1},\tilde u_h^{n+1},{\phi_i}) & =  2(D(\tilde u_h^{n+1}) \tilde u_h^{n+1},{\phi_i}) +  ( (\nabla \cdot \tilde u_h^{n+1})\tilde u_h^{n+1},{\phi_i})  \\
& =  (\tilde u_h^{n+1}\cdot \nabla \tilde u_h^{n+1},{\phi_i}) +   ({\phi_i}\cdot \nabla \tilde u_h^{n+1},\tilde u_h^{n+1})  + ( (\nabla \cdot \tilde u_h^{n+1})\tilde u_h^{n+1},{\phi_i}) \\
& =  - (\tilde u_h^{n+1}\cdot \nabla {\phi_i},\tilde u_h^{n+1})  -  ( (\nabla \cdot \tilde u_h^{n+1})\tilde u_h^{n+1},{\phi_i})  \\
& \ \ \ \ \  -\frac12   ( (\nabla \cdot {\phi_i}) \tilde u_h^{n+1},\tilde u_h^{n+1})  + ( (\nabla \cdot \tilde u_h^{n+1})\tilde u_h^{n+1},{\phi_i}) \\
& =  - (\tilde u_h^{n+1}\cdot \nabla {\phi_i},\tilde u_h^{n+1}).
\end{align*}
For each of $i=1,2,3$, this last term can be seen to vanish by a calculation.  For example taking $i=1$,
\begin{align*}
 (\tilde u_h^{n+1}\cdot \nabla {\phi_1},\tilde u_h^{n+1}) & = \left( \tilde u_h^{n+1} \tilde u_h^{n+1}, \left( \begin{array}{ccc} 0& 0 & 1 \\ 0 & 0 & 0 \\ -1 & 0 & 0 \end{array}  \right) \right) \\
 & = \int_{\Omega} (\tilde u_h^{n+1})_1 (\tilde u_h^{n+1})_3 - (\tilde u_h^{n+1})_3 (\tilde u_h^{n+1})_1 \ dx \\
 & = 0,
\end{align*}
and the cases of $i=2,3$ follow similarly.

Hence for EMAC-BE-PROJ, we obtain the momentum balance
\[
(\tilde M_{X})_i^{n+1} - (\tilde M_{X})_i^n = (M_{X})_i^{n+1} - (M_{X})_i^n = \Delta t  (f (t^{n+1}),{\phi_i} ),
\]
where $(\tilde M_{X})_i^n := (\tilde {u}_h^n, \phi_i )$ and  $(M_{X})_i^n := ({u}_h^n, \phi_i )$, which is the backward Euler analogue of the continuous NSE momentum balance.

For the angular momentum balance in SKEW-BE-PROJ, the same procedure as for the EMAC case can be done except for the nonlinear term, for which we have that
\begin{align*}
b^{\ast}(\tilde u_h^{n+1},\tilde u_h^{n+1},{\phi_i}_h) & = b^{\ast}(\tilde u_h^{n+1},\tilde u_h^{n+1}, \phi_i ) \\
 & = (\tilde u_h^{n+1}\cdot \nabla \tilde u_h^{n+1}, \phi_i )  + \frac12 ( (\nabla \cdot \tilde u_h^{n+1}) \tilde u_h^{n+1}, \phi_i ) \\
 & = - (\tilde u_h^{n+1}\cdot \nabla \phi_i , \tilde u_h^{n+1} )  - \frac12 ( (\nabla \cdot \tilde u_h^{n+1}) \tilde u_h^{n+1}, \phi_i )\\
 & =  - \frac12 ( (\nabla \cdot \tilde u_h^{n+1}) \tilde u_h^{n+1}, \phi_i ) \textcolor{black}{,}
 \end{align*}
 where we use the same calculation as in the EMAC case to get $ (\tilde u_h^{n+1}\cdot \nabla \phi_i , \tilde u_h^{n+1} )=0$.  Again we observe the problem that
 $- \frac12 ( (\nabla \cdot \tilde u_h^{n+1}) \tilde u_h^{n+1}, \phi_i ) \ne 0$ since $ (\nabla \cdot \tilde u_h^{n+1})\ne 0$, and so the angular momentum balance has a contribution from the nonlinear term for SKEW-BE-PROJ:
\[
(M_{X})_i^{n+1} - (M_{X})_i^n = (\tilde M_{X})_i^{n+1} - (\tilde M_{X})_i^n = \Delta t f_i^{n+1} + \frac{\Delta t}{2}  ( (\nabla \cdot \tilde u_h^{n+1}) \tilde u_h^{n+1}, \phi_i ).
\]

\section{Improved convergence estimate for EMAC-BE-PROJ}
\label{sec:Proj_analysis}

We now show that EMAC-BE-PROJ gives longer time accuracy than for SKEW-BE-PROJ.  In particular, we will show that EMAC-BE-PROJ gives an improved constant in the Gronwall exponent in that it has a reduced dependence on the viscosity. There are many different projection methods one can study such as pressure correction, velocity correction and others \cite{GMS06}, and also in multiple steps for various norms as in \cite{shen1992error}, or comparing a fully discrete method to a semi-discrete method as in \cite{BC07}.  For simplicity, we choose to study the most basic method (defined in section 3) and compare it to a smooth true solution.  This proof shows $O(\Delta t^{1/2})$ temporal error in the $L^{\infty}(0,T;L^2)\cap L^2(0,T;H^1)$ norms, with the key distinction of the proof being the difference in the Gronwall constant for EMAC having no explicit dependence on the viscosity, while for SKEW it depends on $\textcolor{black}{\nu}^{-1}$ explicitly.  While other techniques could improve the asymptotic temporal error in other norms, these other variations of convergence proofs all use the Gronwall inequality and thus the same result will be found: with EMAC the Gronwall exponent will be significantly reduced compared to SKEW, which suggests longer time accuracy for EMAC vs. SKEW.

Before considering an error analysis, we first discuss its well-posedness.

\begin{lemma}
The EMAC-BE-PROJ method is unconditionally stable: for any $\Delta t>0$, solutions satisfy
\begin{equation}
   ||\tilde u_h^{M}||^2 + \sum_{n=0}^{M-1}||\tilde u_h^{n+1} - u_h^{n}||^2 + \nu \Delta t \sum_{n=0}^{M-1}||\nabla \tilde u_h^{n+1}||^2
  \leq ||u_0||^2 + \nu^{-1}\Delta t \sum_{n=0}^{M-1}||f(t^{n+1})||_{\blue{X'}}^2. \label{stability_bound}
\end{equation}
Moreover, for any $\Delta t>0$ solutions exist, and provided $\Delta t<O(h^{3})$ solutions are guaranteed to be unique. 
\end{lemma}
\begin{remark}
The same lemma holds for SKEW-BE-PROJ with an almost identical proof.
\end{remark}

\begin{proof}
The stability bound follows immediately from the energy conservation analysis in the previous section.  With this unconditional stability, Leray-Schauder can be used to infer solutions to EMAC Step 1 at any time step, in an analogous way to what is done in \cite{Laytonbook} for steady NSE.  For uniqueness of EMAC Step 1 solutions, suppose there are 2 solutions at time step $n$, $\tilde u_h$ and $\tilde w_h$.  Plugging them into EMAC Step 1, setting $e=\tilde u_h - \tilde w_h$ and subtracting their equations gives
\[
\frac{1}{\Delta t}(e,v_h) + \nu(\nabla e,\nabla v_h) = -c(\tilde u_h,e,v_h) - c(e,\tilde w_h,v_h),
\]
for any $v_h \in X_h$.  Taking $v_h=e$ produces
\begin{align*}
\frac{1}{\Delta t}\| e \|^2 + \nu\| \nabla e \|^2 & = -c(\tilde u_h,e,e) - c(e,\tilde w_h,e) \\
& \le M (\| \nabla \tilde u_h \| + \| \nabla \tilde w_h \| ) \| \nabla e \|^{3/2} \| e \|^{1/2} \\
& \le C\Delta t^{-1/2}\| \nabla e \|^{3/2} \| e \|^{1/2} ,
\end{align*}
thanks to \eqref{cbound}, the bound on solutions \eqref{stability_bound}, with $C$ being independent of $h$ and $\Delta t$.  Now using the inverse inequality and reducing gives
\[
\Delta t^{-1/2}\| e \|^2 \le C h^{-3/2} \| e \|^2,
\]
which implies that $\Delta t < O(h^{3})$ will yield $e=0$ and thus uniqueness of solutions for step 1.  Step 2 is the $L^2$ projection of Step 1 solutions into the divergence free subspace, and so preserves existence and uniqueness of solutions.
\end{proof}

Next we prove an error estimate for EMAC-BE-PROJ.  We use the notation $\| \varphi \|_{p,r} := \| \varphi \|_{L^p(0,T;H^r\textcolor{black}{)}}$.
\begin{theorem}\label{ErrorEstimate}
Let $(\uh,p_h)$ be the solution of Algorithm \ref{beleproj} (EMAC-BE-PROJ), and 
$(u,p)$ be a NSE solution that satisfies \textcolor{black}{$u \in L^{\infty}(0,T;H^{k+1} \cap V)$} with $k\ge 2$, $u_{t} \in L^{\infty}(0,T;H^{k+1})$, $u_{tt} \in L^2(0,T;H^{k+1})$ and $p \in L^2(0,T;H^k)$. Also, let $\be^{n+1}=u^{n+1}-u_h^{n+1}$ and $\tilde \be^{n+1}= u^{n+1}- \tilde u_h^{n+1}$. Then, for sufficiently small $\Delta t$ we have
\begin{eqnarray*}
&&  |||\tilde \be^M|||^2_{\infty,0}  
+  \nu|||\nabla \tilde \be^{n+1}|||^2_{2,0} \leq 
  C K  \bigg( 
\nu^{-1} h^{2k} (\textcolor{black}{ h^2 \norm{u_t}_{2,k+1}^2} + \nu^2T \textcolor{black}{ \trinorm{u}}_{\infty,k+1}^2 \nonumber \\
& & + T \textcolor{black}{\trinorm{ u }}_{\infty,k+1}^4  +  \textcolor{black}{ \trinorm{ p }_{2,k}^2)}
+ \nu^{-1}\Delta t^2 T \textcolor{black}{ \trinorm{u_{tt}}_{\infty,0}} 
+  \delt  \textcolor{black}{\trinorm{ p  }^2_{2,1}}
\bigg),
 \end{eqnarray*}
 where $K=\exp{\left( \Delta t \sum_{n = 0}^{l} \frac{\gamma_{n}}{(1 - \Delta t \, \gamma_{n})} \right)}$ and $\gamma_n = C \| \nabla u^n \|_{L^{\infty}}$.
 \end{theorem}
 
 \begin{remark} \label{remGron}
The key improvement from EMAC is evident in the Gronwall constant K from the theorem.  In particular we note there is no explicit dependence on the viscosity.  For SKEW-BE-PROJ, a nearly identical estimate would be obtained (see \eqref{banal4} below) under the same smoothness assumptions on the true solution, but the Gronwall constant would have $\gamma_n=C(\| \nabla u^n \|_{L^{\infty}}^2 + \nu^{-1} \| u^n \|_{L^{\infty}}^2 )$. This suggests EMAC-BE-PROJ has better longer time accuracy than SKEW-BE-PROJ.
 \end{remark}
 
  \begin{remark}
 Following \cite{shen1992error}, one can obtain an identical error bound but with left hand side in terms of $e^n$ instead of $\tilde e^n$.
  \end{remark}

\begin{proof}
We split the error in the usual way as $\be^{n+1}= \bfeta^{n+1}-\bfphi_h^{n+1}$ and $\tilde \be^{n+1}=\bfeta^{n+1}- \tilde \bfphi_h^{n+1}$ where $\bfeta=u-I(u)$, $\tilde \bfphi_h =I(u)-\tilde u_h\, \in \bX_h$ and $\bfphi_h=I(u)- u_h\, \in \bY_h$ with $I(u)$ being a pointwise div-free interpolant of $u$ in $X_h$.
Subtract \eqref{projemac1} (Step 1 EMAC) from the NSE at time $t^{n+1}$ and tested with $\chi_h\in X_h$ to obtain
\begin{eqnarray*}
  \frac{1}{\Delta t}(\tilde \be^{n+1}- \be_h^n,\bfchi_h)+c( u^{n+1}, u^{n+1},\bfchi_h) - c( \tilde u_h^{n+1},\tilde u_h^{n+1},\bfchi_h) + \nu(\nabla \tilde \be^{n+1},\nabla \bfchi_h)\\
  -(p^{n+1},\nabla \cdot \bfchi_h) =  \left(\frac{1}{\Delta t}(u^{n+1}- u^n)-u_t^{n+1},\bfchi_h\right), \quad \forall \bfchi_h \in \bX_h.
\end{eqnarray*}
Splitting the error, letting $\bfchi_h =\tilde \bfphi_h^{n+1}$, we obtain
\begin{eqnarray}\label{err_1} 
&  &\frac{1}{2\Delta t}\left(||\tilde \bfphi_h^{n+1}||^2-||\bfphi_h^{n}||^2+||\tilde \bfphi_h^{n+1} - \bfphi_h^{n}||^2\right) + \nu||\nabla \tilde \bfphi_h^{n+1}||^2
= \frac{1}{\Delta t}(\bfeta^{n+1}- \bfeta^n,\tilde \bfphi_h^{n+1})\nonumber \\
& & \textcolor{black}{+} \nu(\nabla \bfeta^{n+1},\nabla \tilde \bfphi_h^{n+1})+(u_t^{n+1}-\frac{1}{\Delta t}(u^{n+1}-u^n), \tilde \bfphi_h^{n+1})
- (p^{n+1},\nabla \cdot \tilde \bfphi_h^{n+1}) \\
&& 
\textcolor{black}{+} c( u^{n+1}, u^{n+1},\tilde \bfphi_h^{n+1}) \nonumber 
\textcolor{black}{-} c( \tilde u_h^{n+1},\tilde u_h^{n+1},\tilde \bfphi_h^{n+1}).
\end{eqnarray}

We now bound the above right hand side terms. Other than the nonlinear term, these bounds are fairly standard for NSE finite element error analysis, see e.g. \cite{Laytonbook},
\begin{eqnarray}
\frac{1}{\Delta t}(\bfeta^{n+1}- \bfeta^n,\tilde \bfphi_h^{n+1}) & \le & 
\textcolor{black}{C \frac{\nu\inv}{\delt} \norm{\bfeta^{n+1}- \bfeta^n} + C \frac{\nu}{12} \norm{\tilde \bfphi_h^{n+1}}} \nonumber\\
 & \leq & \textcolor{black}{ C\nu\inv  \omint \pare{ \frac{1}{\delt} \int_{t^n}^{t^{n+1}} \abs{\bfeta_t} \, dt }^2 d\Omega + C \frac{\nu}{12} \norm{\tilde \bfphi_h^{n+1}}} \nonumber\\
&\leq & \textcolor{black}{ C \frac{\nu\inv}{\delt} \int_{t^n}^{t^{n+1}} \norm{\bfeta_t}^2 \, dt  + C \frac{\nu}{12} \norm{\tilde \bfphi_h^{n+1}}} \label{t1_bound}\\
\nu(\nabla \bfeta^{n+1},\nabla \tilde \bfphi_h^{n+1}) & \leq & \frac{\nu}{12}||\nabla \tilde \bfphi_h^{n+1}||^2+C\nu||\nabla \bfeta^{n+1}||^2, \label{t2_bound} \\
 (u_t^{n+1}-\frac{1}{\Delta t}(u^{n+1}-u^n),\tilde \bfphi_h^{n+1}) & \leq & \frac{\nu}{12}||\nabla \tilde \bfphi_h^{n+1}||^2+C\nu^{-1}\Delta t^2||u_{tt} (t^*)||^2, \label{t3_bound} 
 \end{eqnarray}
 with $t^n \le t^* \le t^{n+1}$.  The pressure term takes a few additional steps, and we denote by $ I_h $ the nodal interpolant of $ p $ in $ Q_h $ and use Cauchy-Schwarz and Young's inequalities along with an interpolation bound to get
 
 \begin{align}
    | (p^{n+1} & , \Grad \cdot \Tilde{\phi}_h^{n+1} ) |   =   \abs{ ( p^{n+1} - I_h(p(t^{n+1})), \Grad \cdot \Tilde \bfphi_h^{n+1}  ) + ( I_h(p^{n+1}), \Grad \cdot \Tilde  \bfphi_h^{n+1} ) } \nonumber \\
 &  \le  \abs{ \textcolor{black}{(} p^{n+1} - I_h ( p^{n+1}) ,\Grad \cdot \Tilde  \bfphi_h^{n+1} \textcolor{black}{)}} + \abs{ (  I_h(p^{n+1}), \Grad \cdot ( \Tilde{\phi}_h^{n+1} - \bfphi_h^{n+1} )   ) }  \nonumber \\
 & \leq   C \norm{ p^{n+1} - I_h ( p^{n+1} ) }\norm{\Grad \Tilde  \bfphi_h^{n+1} } + \abs{ (  I_h(p^{n+1}), \Grad \cdot ( \Tilde{\phi}_h^{n+1} - \bfphi_h^{n+1} )   ) }  \nonumber \\
 & \leq   \frac{\nu}{12} \norm{ \Grad \Tilde{\phi}_h^{n+1} }^2 + C \nu^{-1} \norm{ p^{n+1} - I_h ( p^{n+1} ) }^2 \nonumber \\
 & \qquad + \frac{1}{4 \delt} \norm{ \Tilde{\phi}_h^{n+1} - \bfphi_h^{n+1} }^2 + C \delt \norm{ \Grad p^{n+1} }^2   \nonumber \\
& \leq  \frac{\nu}{12} \norm{ \Grad \Tilde{\phi}_h^{n+1}  }^2 + C\nu^{-1} h^{2k} \norm{p^{n+1}}_{H^{k}}^2 \nonumber\\
& \qquad + \frac{1}{4 \delt} \norm{ \Tilde{ \phi}_h^{n+1} - \bfphi_h^{n+1} }^2 + C \delt \norm{ \Grad p^{n+1}  }^2. \label{t7_bound} 
\end{align}

For the nonlinear terms, we first decompose following \cite{OR20} to obtain
\begin{align*}
c(u^{n+1},u^{n+1},\tilde \bfphi_h^{n+1}) &- c(\tilde u_h^{n+1},\tilde u_h^{n+1},\tilde \bfphi_h^{n+1})
   = ([2\bD(\bfeta^{n+1})+\textcolor{black}{\divergence}\bfeta^{n+1}]u^{n+1},\tilde \bfphi_h^{n+1}) \\
   &+ ([2\bD(\textcolor{black}{\stokes{u^{n+1}}})+\textcolor{black}{\divergence \stokes{u^{n+1}}}]\bfeta^{n+1},\tilde \bfphi_h^{n+1})\\
   &- ([\bD(\textcolor{black}{\stokes{u^{n+1}}})+\frac12 \textcolor{black}{\divergence \stokes{u^{n+1}}}] \tilde \bfphi_h^{n+1},\tilde \bfphi_h^{n+1}).
\end{align*}

Next we estimate the terms on the right-hand side of the above equations. We
repeat the estimates similar to above and use the  $H^1$-stability of the Stokes interpolant to get
\bea\label{ss6-7}
  & & \hspace{-1cm} |([2\bD(\bfeta^{n+1})+\textcolor{black}{\divergence}\bfeta^{n+1}]u^{n+1}, \tilde \bfphi_h^{n+1}) + ([2\bD(\textcolor{black}{\stokes{u^{n+1}}})+\textcolor{black}{\divergence\stokes{u^{n+1}}}]\bfeta^{n+1}, \tilde \bfphi_h^{n+1})| \nonumber\\
  & \le & \hspace{-.25cm}  C\nu^{-1} \| \nabla u^{n+1} \|^2 \|  \bfeta^{n+1} \| \| \nabla \bfeta^{n+1} \|
  +  C\nu^{-1} \| \nabla u^{n+1} \| \|  u^{n+1}\| \| \nabla \bfeta^{n+1} \|^2 +  \textcolor{black}{\frac{\nu}{6}} \| \nabla \tilde \bfphi_h^{n+1} \|^2, \qquad \;
\eea
and
\begin{align}
 |([\bD(\textcolor{black}{\stokes{u^{n+1}}})+\frac12\textcolor{black}{\divergence\stokes{u^{n+1}}}]\tilde \bfphi_h^{n+1},\tilde \bfphi_h^{n+1})| & \le \frac32\|\bD(\textcolor{black}{\stokes{u^{n+1}}}\|_{L^{\infty}}\|\tilde \bfphi_h^{n+1} \|_{L^2}\|\tilde \bfphi_h^{n+1} \|_{L^2}\nonumber \\ & \le
 C\|\nabla u^{n+1}\|_{L^{\infty}}\|\tilde \bfphi_h^{n+1}\|^2,\label{ss8emac}
 \end{align}
thanks to the stability of the Stokes projection in \eqref{eq:Stokes_approx} with $r=\infty$.

By adding and subtracting our interpolant $I(u)$ into Step 2 we obtain
\begin{equation*}
  \frac{1}{\Delta t}({\phi}_h^{n+1}-\tilde {\phi}_h^{n+1},\bw_h) = - (\textcolor{black}{P}_h^{n+1},\nabla \cdot \bw_h), \quad \forall \bw_h \in \bY_h.
\end{equation*}
Let $\bw_h=\bfphi_h^{n+1}$ and using that $ \bfphi_h^{n+1}$ is weakly divergence free we have 
\begin{equation*}
  \frac{1}{\Delta t}({\phi}_h^{n+1}-\tilde {\phi}_h^{n+1},\bfphi_h^{n+1}) = - (\textcolor{black}{P}_h^{n+1},\nabla \cdot \bfphi_h^{n+1})=0,
\end{equation*}
and thus
\begin{equation*}
  \frac{1}{2 \Delta t}\left(||\bfphi_h^{n+1}||^2-||\tilde \bfphi_h^{n+1}||^2 +||\bfphi_h^{n+1}- \tilde \bfphi_h^{n+1}||^2\right) =0.
\end{equation*}
The above implies
\begin{equation*} \label{error_2}
||\tilde \bfphi_h^{n+1}||^2 = ||\bfphi_h^{n+1}||^2 + || \bfphi_h^{n+1} - \tilde \bfphi_h^{n+1} ||^2,
\end{equation*}
which will be used in (\ref{err_1}) to obtain a telescoping sum. Also, substituting the obtained bounds (\ref{t1_bound})-(\ref{ss8emac}) into (\ref{err_1}), multiplying by $2 \Delta t$ and summing up from $n=0$ to $M-1$ (with the assumption that $||\bfphi_h^0||=0$) gives us the following bound
\begin{eqnarray}
& & ||\tilde \bfphi_h^{M}||^2 + \sum_{n=0}^{M-1} \left( \frac{1}{2}||\tilde \bfphi_h^{n+1} - \bfphi_h^{n}||^2 + ||\bfphi_h^{n+1}- \tilde \bfphi_h^{n+1}||^2 \right)
+\nu \Delta t \sum_{n=1}^{M} ||\nabla \tilde \bfphi_h^{n}||^2 \nonumber \\
&\leq& \textcolor{black}{C \nu\inv \int_{0}^{T} \norm{\bfeta_t}^2 \, dt}
+ C\nu \Delta t \sum_{n=1}^M ||\nabla \bfeta^{n}||^2 + C \textcolor{black}{\delt^2 \sum_{n=1}^M \norm{ \Grad p^{n}  }^2} \nonumber \\
&& + C\nu^{-1}\Delta t^2 T \textcolor{black}{\trinorm{u_{tt}}_{L^{\infty}(0,T;L^2)}^2}
+ C\nu^{-1} h^{2k} \textcolor{black}{ \textcolor{black}{\trinorm{ p}_{L^{2}(0,T;H^k)}^2} }  \nonumber \\
& &
+C\nu^{-1}\Delta t \sum_{n=1}^M  \| \nabla u^{n} \|^2 \| \nabla \bfeta^{n} \|^2  + 
C \Delta t \sum_{n=1}^M \|\nabla u^{n}\|_{L^{\infty}}\|\tilde \bfphi_h^{n}\|_{L^2}^2 \textcolor{black}{,}
\label{error_3}
\end{eqnarray}
thanks also to the Poincare inequality and reducing to get the second to last term.  We reduce the right hand side further, utilizing interpolation estimates and true solution regularity, and dropping positive left hand side terms to find that
\begin{eqnarray}
& & ||\tilde \bfphi_h^{M}||^2 
+\nu \Delta t \sum_{n=1}^{M} ||\nabla \tilde \bfphi_h^{n+1}||^2 \nonumber \\
&\leq& 
C\nu^{-1} h^{2k} ( h^2 \textcolor{black}{\norm{u_t}_{2,k+1}^2}   + \nu^2T \textcolor{black}{\trinorm{u}_{\infty,k+1}^2} + T \textcolor{black}{\trinorm{u}_{\infty,k+1}^4} + \textcolor{black}{\trinorm{ p }_{2,k}^2})
 \nonumber \\
&& + C\nu^{-1}\Delta t^2 T \textcolor{black}{\trinorm{u_{tt}}}_{\infty,0}^{\textcolor{black}{2} }
+ C \textcolor{black}{\Delta t \textcolor{black}{\trinorm{ p  }}^2_{2,1}}
+C \Delta t \sum_{n=1}^M \|\nabla u^{n}\|_{L^{\infty}}\|\tilde \bfphi_h^{n}\|_{L^2}^2.
\label{error_3}
\end{eqnarray}
Now by the Gronwall inequality with $\Delta t$ sufficiently small, i.e. $\gamma_n \Delta t := C ||\nabla u^n||_{L^{\infty}} \Delta t < 1$,  we obtain
\bea
& & ||\tilde \bfphi_h^{M}||^2 
+\nu \Delta t \sum_{n=0}^{M-1} ||\nabla \tilde \bfphi_h^{n+1}||^2 \nonumber \\
&\leq& C \exp{\left( \Delta t \sum_{n = 0}^{l} \frac{\gamma_{n}}{(1 - \Delta t \, \gamma_{n})} \right)}  \bigg( 
\nu^{-1} h^{2k} ( T \textcolor{black}{\norm{u_t}_{2,k+1}^2}  + \nu^2T \textcolor{black}{ \trinorm{ u }_{\infty,k+1}^2} \nonumber \\
& & + T \textcolor{black}{\trinorm{ u }_{\infty,k+1}^4}  + \textcolor{black}{ \trinorm{ p }_{2,k}^2})
+ \nu^{-1}\Delta t^2 T \textcolor{black}{ \trinorm{u_{tt}}^2_{\infty,0} }
+ \delt \textcolor{black}{ \trinorm{ p }^2_{2,1}} \bigg). \label{error_5}
\eea
From here, the triangle inequality finishes the proof.
\end{proof}

A similar proof for SKEW-BE-PROJ would follow the same way, except for the nonlinear terms.  In this case, we would expand the difference to get
\begin{align}
b^*(u^{n+1},u^{n+1},\tilde \bfphi_h^{n+1}) & -b^*(\tilde u_h^{n+1},\tilde u_h^{n+1},\tilde \bfphi_h^{n+1}) \nonumber \\
   &= b^*(\tilde u_h^{n+1},\tilde e^{n+1},\tilde \bfphi_h^{n+1})+ b^*(\tilde e^{n+1}, u^{n+1},\tilde \bfphi_h^{n+1})  \nonumber \\
   &= b^*(\tilde u_h^{n+1},\eta^{n+1},\tilde \bfphi_h^{n+1})+ b^*(\tilde e^{n+1}, u^{n+1},\tilde \bfphi_h^{n+1}) \textcolor{black}{.} \label{banal}
\end{align}
For the first term, H\"older, Sobolev inequalities and Young's inequality produce
\begin{align}
    b^*(\tilde u_h^{n+1},\eta^{n+1},\tilde \bfphi_h^{n+1}) & \le \textcolor{black}{\frac{\nu}{12}||\nabla \tilde \bfphi_h^{n+1}||^2 + C\nu^{-1}||\tilde \uh^{n+1}|| ||\nabla \tilde \uh^{n+1}|| ||\nabla \bfeta^{n+1}||^2.}  \label{banal2}
\end{align}
and for the second term
\begin{align}
    b^*(\tilde e^{n+1}, u^{n+1},\tilde \bfphi_h^{n+1})
    & = b^*( \eta^{n+1}, u^{n+1},\tilde \bfphi_h^{n+1})+b^*(\tilde \phi_h^{n+1}, u^{n+1},\tilde \bfphi_h^{n+1})
    .\label{banal3}
\end{align}
The first term on the right hand side of \eqref{banal3} is handled as in \eqref{banal2}, and for the second we employ H\"older, Sobolev and Young inequalities to find
\begin{align}
    b^*(\tilde \phi_h^{n+1}, u^{n+1},\tilde \bfphi_h^{n+1})
    & = (\tilde \phi_h^{n+1}\cdot \nabla u^{n+1},\tilde \bfphi_h^{n+1}) + \textcolor{black}{\frac{1}{2}} ( (\nabla \cdot \tilde \phi_h^{n+1}) u^{n+1},\tilde \bfphi_h^{n+1}) \nonumber \\
    & \le C \| \nabla u^{n+1} \|_{L^{\infty}} \|  \tilde \phi_h^{n+1} \|^2 + C \| \nabla \tilde \phi_h^{n+1} \| \| u^{n+1} \|_{L^{\infty}} \| \tilde \phi_h^{n+1}  \|  \nonumber \\
        & \le C (\| \nabla u^{n+1} \|_{L^{\infty}} + \nu^{-1} \| u^{n+1} \|_{L^{\infty}}^2 ) \|  \tilde \phi_h^{n+1} \|^2 + \textcolor{black}{\frac{\nu}{12}}  \| \nabla \tilde \phi_h^{n+1} \|^2.   \label{banal4}
\end{align}

\textcolor{black}{The above equation requires the use of the Gronwall inequality with $\Delta t$ sufficiently small, i.e. $\gamma_n \Delta t := C (||\nabla u^n||_{L^{\infty}}+\nu^{-1} \| u^{n} \|_{L^{\infty}}^2) \Delta t < 1$, which gives a dependence on $\nu^{-1}$ as stated in Remark \ref{remGron}.}

\section{Extension to coupled schemes}
\label{sec:CN_Analysis}

While the analysis above was to compare EMAC and SKEW for projection methods, the key difference was in the treatment of the nonlinear terms.  Hence, these longer time accuracy results for EMAC can be transferred to standard coupled schemes using mixed finite elements.  In this section, we consider convergence analysis of the Crank-Nicolson FEM, using both the EMAC and SKEW forms of the nonlinearity.

The Crank-Nicolson FEM scheme for EMAC is as follows: Find $(u_h^{n+1},p_h^{n+1})\in (X_h,Q_h)\times (0,T]$ satisfying for all $(v_h,q_h) \in (X_h,Q_h)$,
\begin{align}
\lip{\frac{u^{n+1}_h-u^n_h}{\Delta t}}{v_h}&+c(u_h\chalf,u_h\chalf,v_h)+\nu\lip{\nabla u_h\chalf}{\nabla v_h}\notag\\
&-\lip{p\chalf_h}{\divergence v_h}=\lip{f\pare{t\chalf}}{v_h},\label{eq:EMAC1}\\
\lip{\divergence u_h\chalf}{q_h}&=0.\label{eq:EMAC2}
\end{align}
For SKEW, the scheme is the same but replace $b^*$ with $c$ in the nonlinear term.

We now state a convergence theorem for the CN-FEMs for EMAC and SKEW.  

\begin{theorem}
(i) [SKEW CN-FEM convergence] Let $(u_h^{n+1},p_h^{n+1})$ solve \eqref{eq:EMAC1}-\eqref{eq:EMAC2} with SKEW nonlinearity, and $(u^{n+1},p^{n+1})$ be a NSE solution with $u_t^{n+1}\in X'$, $u^{n+1} \in H^3(\Omega)$, and $p^{n+1}\in H^2(\Omega)$, for $0 \leq n \leq M$.  Denote $e^n=u^n-u_h^n$, $\eta^n=u^n-\stokes{u^n}$, and $\phi_h^n=\stokes{u^n}-u^n_h$.  Then for all $0\leq n \leq M$, the following holds:
\begin{align*}
&\norm{e^M}^2+\nu\Delta t\sum_{n=0}^{M-1}\norm{\nabla e\chalf}^2\\
\leq & \exp\pare{C\Delta t\sum_{n=0}^{M-1}\pare{\norm{\nabla u\chalf}_{L^{\infty}}+\nu\inv\norm{u\chalf}^2_{L^{\infty}}+\frac{3}{2}}}F(\Delta t, h)\\
&+ C\nu(\Delta t)^4\trinorm{ \nabla u_{tt} }^2_{2,0}+C\nu h^{2k}\trinorm{ u}_{2,k+1},
\end{align*}
where
\begin{align*}
F(\Delta t, h)=&C\nu\inv h^{2k+1}\pare{\trinorm{u}^4_{4,k+1}+\norm{\:|\nabla u|\:}^4_{4,0}}\\
&+C\nu\inv h^{2k}\pare{\trinorm{u}^4_{4,k+1}+\nu\inv \pare{\norm{u_h}^2+\nu\inv \trinorm{f}^2_{2,X'}}}\\
&+C\nu\inv\pare{h^{2s+2}\trinorm{p_{\frac{1}{2}}}^2_{2,s+1}+(\Delta t)^4\trinorm{p_{tt}}^2_{2,0}}\\
&+Ch^{2k+2}\trinorm{u_t}^2_{2,k+1}+C\Delta t h^{2k+2}\norm{u_{tt}}^2_{L^2(0,T;H^{k+1})}\\
&+C(\Delta t)^4(\trinorm{u_{ttt}}_{2,0}^2 + \nu\inv\trinorm{p_{tt}}_{2,0}^2 + \trinorm{f_{tt}}_{2,0}^2 + \nu\trinorm{\nabla u_{tt}}_{2,0}^2\\
&+\nu\inv\trinorm{\nabla u_{tt}}_{4,0}^4+\nu\inv\trinorm{ \nabla u}^4_{4,0} + \nu\inv\trinorm{\nabla u_{1/2}}^4_{4,0}).
\end{align*}
(ii) [EMAC CN-FEM convergence] Let $(u_h^{n+1},p_h^{n+1})$ solve \eqref{eq:EMAC1}-\eqref{eq:EMAC2} with the EMAC nonlinearity, and under the same assumptions as part (i).  Then for all $0 \leq n \leq M$, the following holds:
\begin{align*}
& \norm{e^M}^2+\nu\Delta t\sum_{n=0}^{M-1}\norm{\nabla e\chalf}^2
\leq \exp\pare{C\Delta t\sum_{n=0}^{M-1}\pare{\norm{\nabla u\chalf}_{L^{\infty}}}}G(\Delta t, h;P)\\
&+ C\nu(\Delta t)^4\trinorm{\nabla u_{tt}}^2_{2,0}+C\nu h^{2k}\trinorm{u}^2_{2,k+1},
\end{align*}
where
\begin{align*}
G(\Delta t, h\textcolor{black}{;P})&=Ch^{2k}\pare{\trinorm{u}^2_{2,\infty}+\trinorm{u}^4_{4,k+1}}\\
&+Ch^{2k+1}\pare{\trinorm{\nabla u}^2_{2,\infty}+\trinorm{u}^4_{4,k+1}}\\
&+C\nu\inv\pare{h^{2s+2}\trinorm{ P_{\frac{1}{2}}}^2_{2,s+1}+(\Delta t)^4\trinorm{P_{tt}}^2_{2,0}}\\
&+Ch^{2k+2}\trinorm{ u_t}^2_{2,k+1}+C\Delta t h^{2k+2}\norm{u_{tt}}^2_{L^2(0,T;H^{k+1})}\\
&+C(\Delta t)^4(\trinorm{u_{ttt}}_{2,0}^2 + \nu\inv\trinorm{ P_{tt} }_{2,0}^2+\trinorm{f_{tt} }_{2,0}^2+\nu\trinorm{\nabla u_{tt} }_{2,0}^2\\
&+\nu\inv\trinorm{\nabla u_{tt}}_{4,0}^4+\nu\inv\trinorm{ \nabla u }^4_{4,0}+\nu\inv\trinorm{\nabla u_{1/2}}^4_{4,0}).
\end{align*}
\label{thm:error_thm}
\end{theorem}

\begin{proof}
The result for SKEW is known \cite{J16,Laytonbook} while the result for EMAC follows in the same manner as for SKEW in these references but with the nonlinearity treatment from the previous section, \textcolor{black}{particularly, the proof for Theorem \ref{ErrorEstimate}.  We apply the same decomposition, which follows \cite{OR20}, and then the same bounds \eqref{ss6-7}-\eqref{ss8emac}.  The rest of the proof follows identically to the proof for SKEW.}
\end{proof}

\begin{remark}
Just as in the projection method case, the asymptotic error is the same but the key difference between EMAC and SKEW accuracy for the Crank-Nicolson FEM is the reduced Gronwall constant of EMAC - which here does not explicitly depend on the inverse of the viscosity.
\end{remark}

\section{Numerical tests}
\label{sec:Num_Tests}

In this section, we provide numerical results that reinforce the strengths of EMAC over SKEW, in particular the longer time accuracy of EMAC suggested by the better Gronwall constant in its convergence analysis.  
We test both coupled schemes for SKEW and EMAC, as well as projection methods.  We use Freefem++ \cite{hecht} to perform these simulation.  Newton iterations are used to resolve the nonlinearities.  \textcolor{black}{While our analysis considered first order projection methods for the sake of simplicity, for our numerical tests we will additionally test with the second order `rotational' projection method defined as follows:}
\\
\\
Step 1 RotProjB-EMAC: Find $\tilde{u}_h^{n+1}\in X_h$ satisfying, for all $\bfchi_h \in X_h$,
\begin{align}
\begin{cases}
&\frac{1}{\Delta t}\lip{\tilde{u}_h^{n+1}-u_h^n}{\bfchi_h}+\frac{1}{4}c(\tilde{u}_h^{n+1}+u_h^n,\tilde{u}^{n+1}+u_h^n,\bfchi_h)+\frac{\nu}{2}\lip{\nabla(\tilde{u}_h^{n+1}+u^n)}{\nabla\bfchi_h}\\
&\quad -\lip{p_h^n}{\divergence \bfchi_h}=\lip{f\chalf}{\bfchi_h},\\
&\tilde{u}_h^{n+1}=0 \text{ on } \partial \Omega.
\end{cases}
\label{eq:RotProjB-EMAC}
\end{align}
Step 1 RotProjB-SKEW is the same as Step 1 RotProjB-EMAC expect the nonlinear term.

Step 2: Solve for $\phi^{n+1}$
\begin{align}
\begin{cases}
&\frac{-\divergence \tilde{u}_h^{n+1}}{\Delta t}+\Delta \phi_h^{n+1}=0,\\
&u^{n+1}_h\cdot n=0 \text{ on } \partial \Omega.
\end{cases}
\label{eq:RotProjPart2}
\end{align}
From here, we first recover $p_h^{n+1}$ from $\phi_h^{n+1}=\frac{1}{2}\pare{p_h^{n+1}-p_h^n+\nu\divergence \tilde{u}_h^{n+1}}$, and then recover $u_h^{n+1}$ from $\tilde u_h^{n+1}$ and $p_h^{n+1}$ using the projection \cite{GMS06}.  \textcolor{black}{This formulation is a second order method from (3.6)-(3.8) in \cite{GMS06}, equipped with the EMAC nonlinearity formulation.}

\subsection{Planar Lattice Flow}

We first consider an investigation of the evolution of an initial velocity and flow of four vortices which are rotating opposite to one another.  This particular phenomenon, named "planar lattice flow", is a solution to the stationary incompressible Euler equation and has been studied in detail in \cite{SL18,SL17,SLLL18}.  Let $x\in \Omega=(0,1)^2$ and define the initial velocity and true velocity/pressure pair $(u,p)$ as
\begin{align*}
u_0(x)&=\bmat{\sin(2\pi x_1)\sin(2\pi x_2)\\ \cos(2\pi x_1)\cos(2\pi x_2)},\\
u(t,x)&=u_0(x)e^{-8\pi^2\nu t},\\
p(t,x)&=\frac{1}{4}\brak{\cos(4\pi x_1)-\cos(4\pi x_2)}e^{-16\pi^2\nu t}.
\end{align*}

Periodic boundary conditions are imposed on $\partial \Omega$ and we enforce the integral zero-mean condition on the pressure.  We do not impose an external force, so $f=0$, and we set $\nu=4\times 10^{-6}$.  \textcolor{black}{Crank-Nicolson time stepping was used for the coupled schemes (not including SKEW-BE-PROJ and EMAC-BE-PROJ, which use Backward Euler),} and we set $\Delta t=.001$ with the end time $T=5$ for all methods.  A uniform mesh with $(P_2,P_1)$ Taylor hood elements was used with a mesh width of $h=\frac{1}{48}$.  Figure \ref{fig:PLF_init} depicts the initial velocity of the problem.

\begin{figure}[h!]
\centering
\includegraphics[scale=.3]{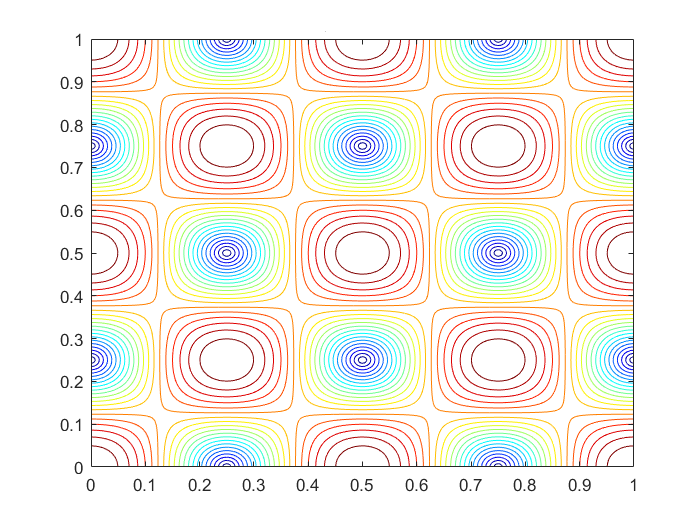}
\caption{Shown above is the initial velocity $u_0$ for planar lattice flow.}
\label{fig:PLF_init}
\end{figure}

The true solution of this problem is for it to decay exponentially with time, but to remain stationary in space.
\textcolor{black}{As time goes on, the term $e^{-8\pi^2\nu t}$ will uniformly decay the initial condition.  Plots of computed solutions at t=5 are shown in figure \ref{fig:PLF}. For each formulation using the SKEW nonlinear term, we observe oscillations to the point where do not see anything that resembles the correct solution.  However, the EMAC formulations strongly resemble Figure \ref{fig:PLF_init}, although some error is clearly present.}  

\textcolor{black}{Figure \ref{fig:PLF_error} shows a semilog plot of $L^2$ error at every timestep for each of the six methods, and as expected we see much better performance from the EMAC methods over the SKEW methods.  Specifically, we notice nearly identical error for EMAC and RotProjB-EMAC, where EMAC-BE-PROJ performs slightly worse.  This is expected because it is a first order method, so it naturally will not outperform EMAC and RotProjB-EMAC.  There does not seem to be a huge difference in $L^2$ error otherwise.  The SKEW methods do not perform well at all over time, and we observe a very large $L^2$ error which level off at $O(10)$ only due to the stability of the method.}

\begin{figure}[h!]
\centering
\includegraphics[scale=.25]{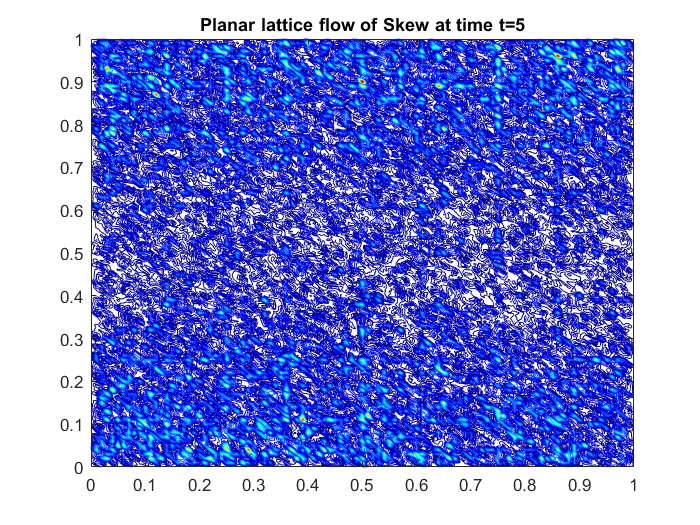}
\includegraphics[scale=.25]{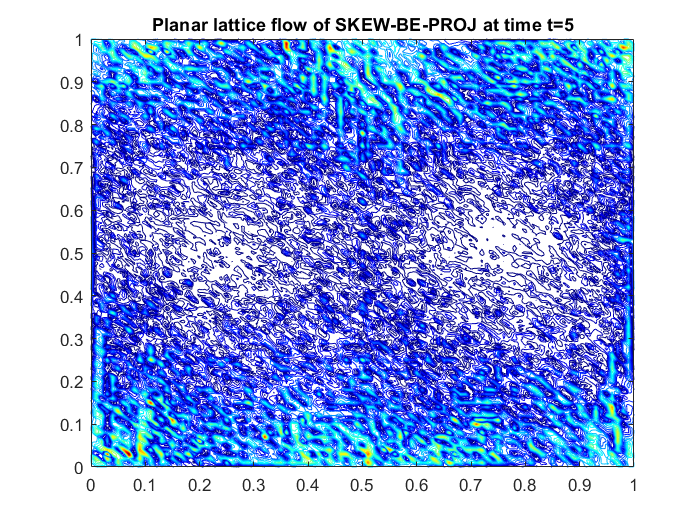}
\includegraphics[scale=.25]{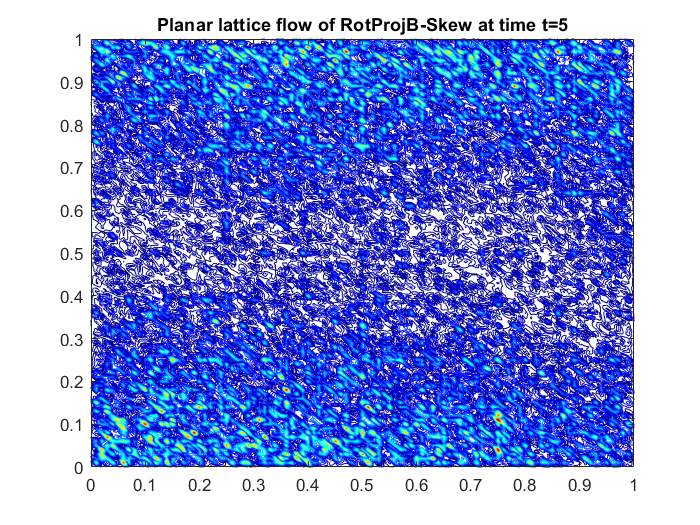}\\
\includegraphics[scale=.25]{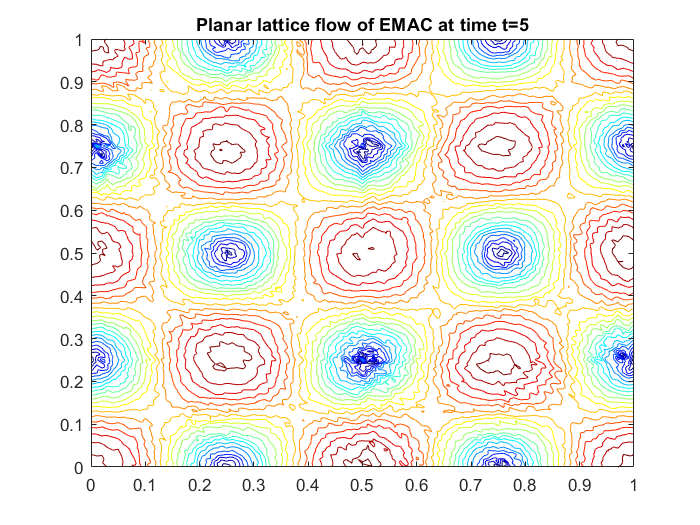}
\includegraphics[scale=.25]{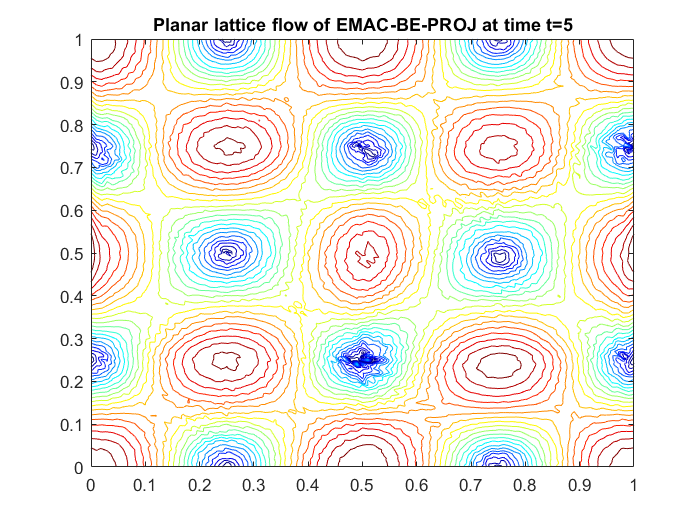}
\includegraphics[scale=.25]{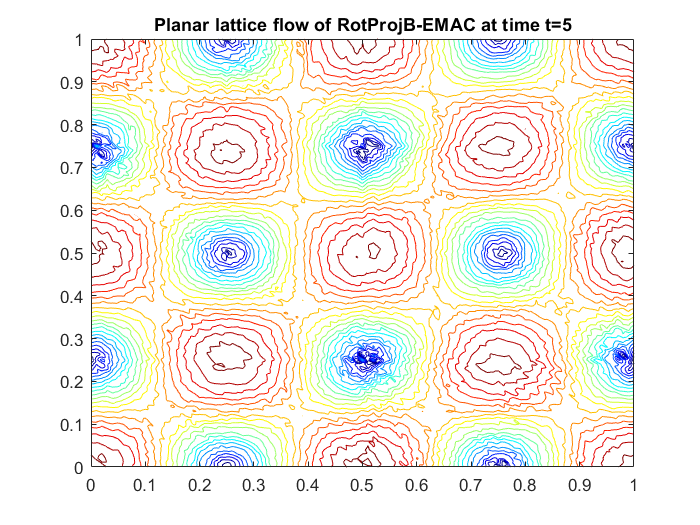}\\
\caption{Plots of the solution of each formulation at time $t=5$}
\label{fig:PLF}
\end{figure}

\begin{figure}[h!]
\centering
\includegraphics[scale=.4]{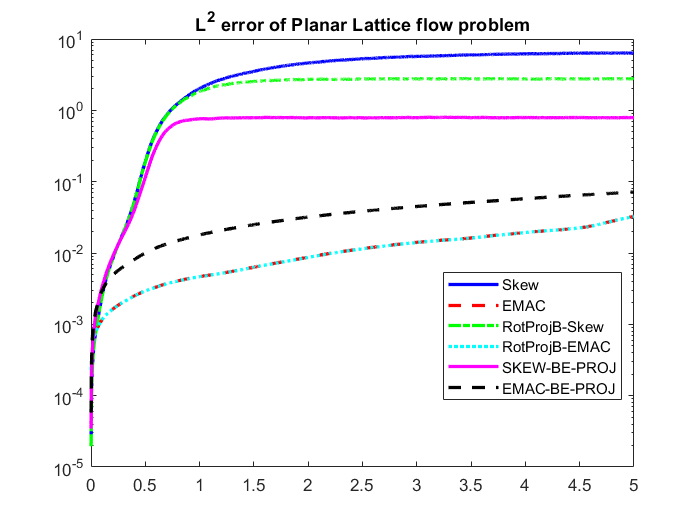}
\caption{Shown above $L^2$ error for each formulation vs. time for the planar lattice vortex problem.}
\label{fig:PLF_error}
\end{figure}

\subsection{Gresho problem}

We next test the methods on the Gresho standing vortex problems.  Computations involving coupled schemes using SKEW and EMAC are well documented \cite{CHOR17, OR20}, but we include them here with projection method results for the sake of comparison.  The initial velocity is a solution to the steady Euler equations, so it is stationary in time, which makes it easy to measure the effectiveness of a formulation.  It will also help measure how effective a method is in conservation properties, since we assume $f=0$ and $\nu=0$.

We begin by defining $r=\sqrt{x^2+y^2}$ on $\Omega=(-0.5,0.5)^2$.  The velocity and pressure are defined as
\begin{align*}
&u=\begin{cases}
\bmat{-5y\\5x} &\text{for }r< 2,\\
\bmat{\frac{2y}{r}+5y\\\frac{2x}{r}-5x} &\text{for }.2 \leq r\leq .4,\\
\bmat{0\\0} &\text{for }r>.4,
\end{cases}\\
&p=\begin{cases}
12.5r^2+C_1 &\text{for }r<.2,\\
12.5r^2-20r+4\log(r)+C_2 &\text{for }.2 \leq r \leq .4,\\
0 &\text{for }r>.4,
\end{cases}
\end{align*}
where
\begin{align*}
&C_2=-12.5(.4)^2+20(.4)^2-4\log(.4),\\
&C_1=C_2-20(.2)+4\log(.2).
\end{align*}

We again use Crank-Nicolson time stepping for the coupled schemes \textcolor{black}{and Backward Euler for SKEW-BE-PROJ and EMAC-BE-PROJ}.  We solve using $\Delta t=.01$ and set $T=4$.  Taylor-Hood $(P_2,P_1)$ elements were used with a mesh size of $h=\frac{1}{48}$.  Similar to the planar lattice flow problem in the previous section, we have an initial velocity (shown in figure \ref{fig:True_Gresho}) that we compare to the computed solutions at later times.  The difference is that the vortex should maintain its shape because it is a solution to a steady-state problem.  However, numerical errors are inevitable, and accumulate after many iterations.  Previous work in \cite{CHOR17, I21} shows that EMAC tends to outperform most conventional formulations over time, and we will again observe this phenomena for this test problem.

\begin{figure}[h!]
\centering
\includegraphics[scale=.4]{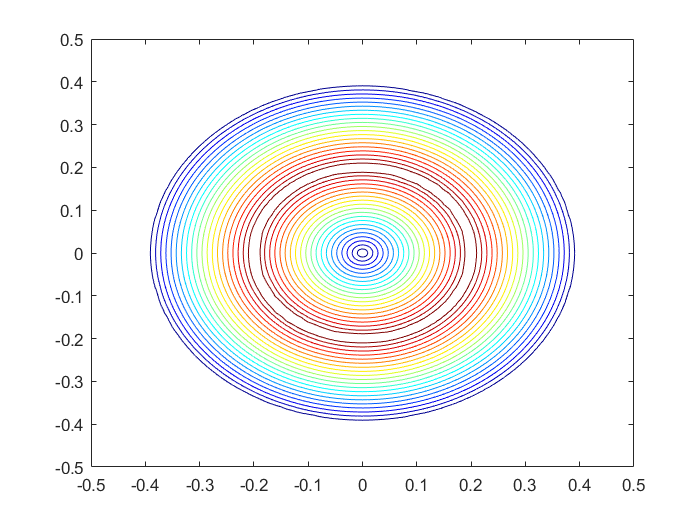}
\caption{Initial velocity for the Gresho problem.}
\label{fig:True_Gresho}
\end{figure}

\begin{figure}[h!]
\centering
\begin{footnotesize}
\end{footnotesize}
\includegraphics[scale=.16]{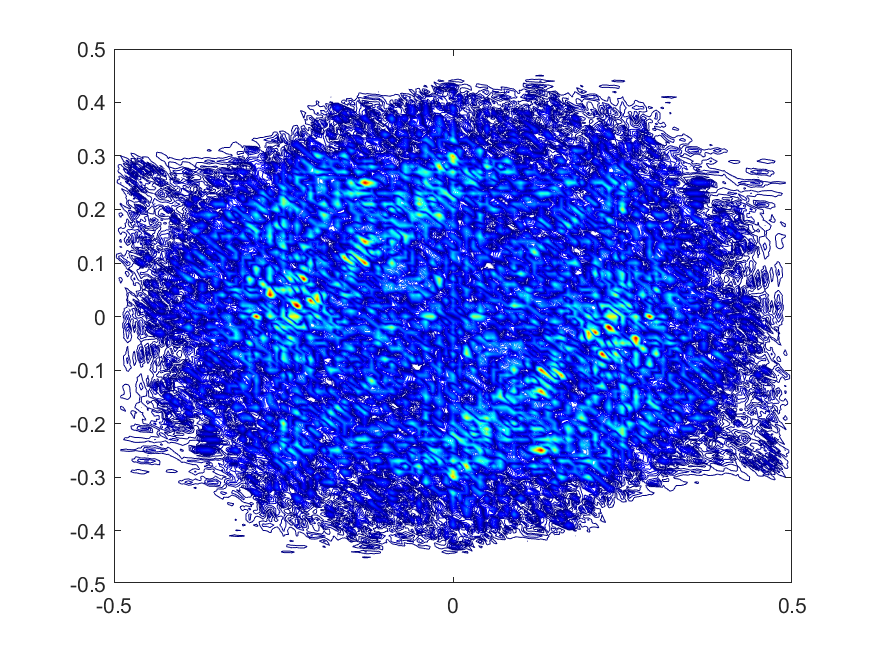}
\includegraphics[scale=.16]{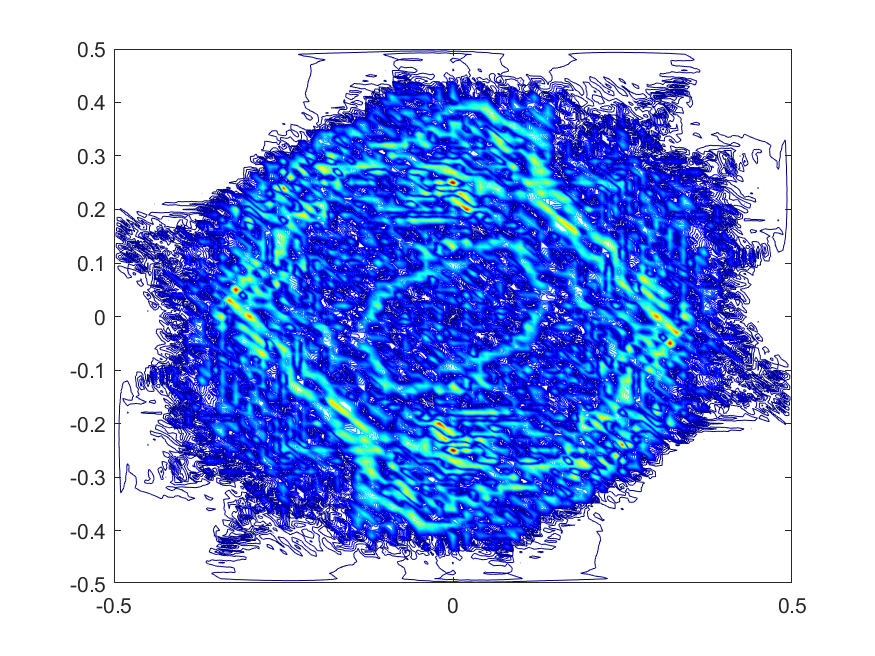}
\includegraphics[scale=.16]{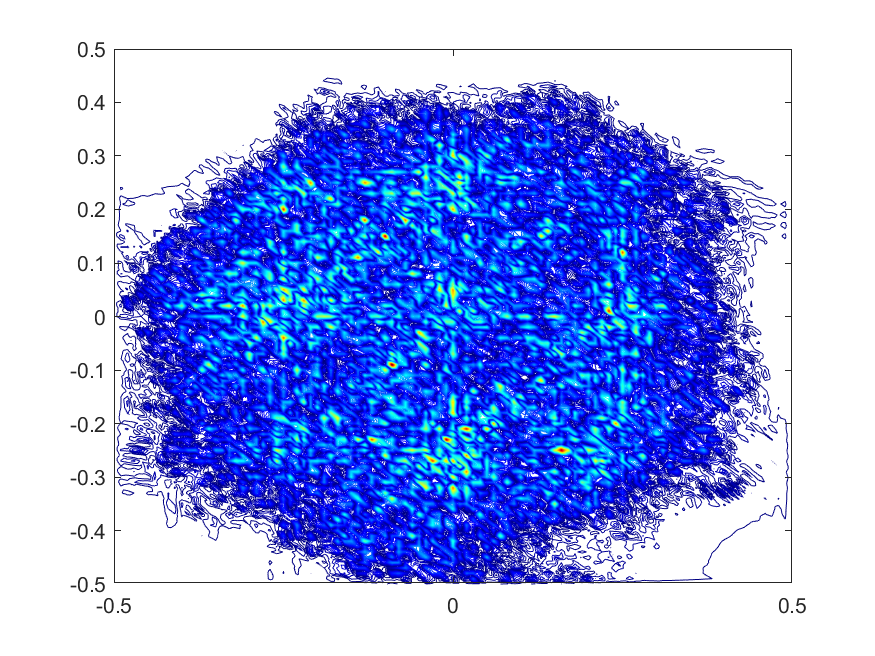}
\includegraphics[scale=.16]{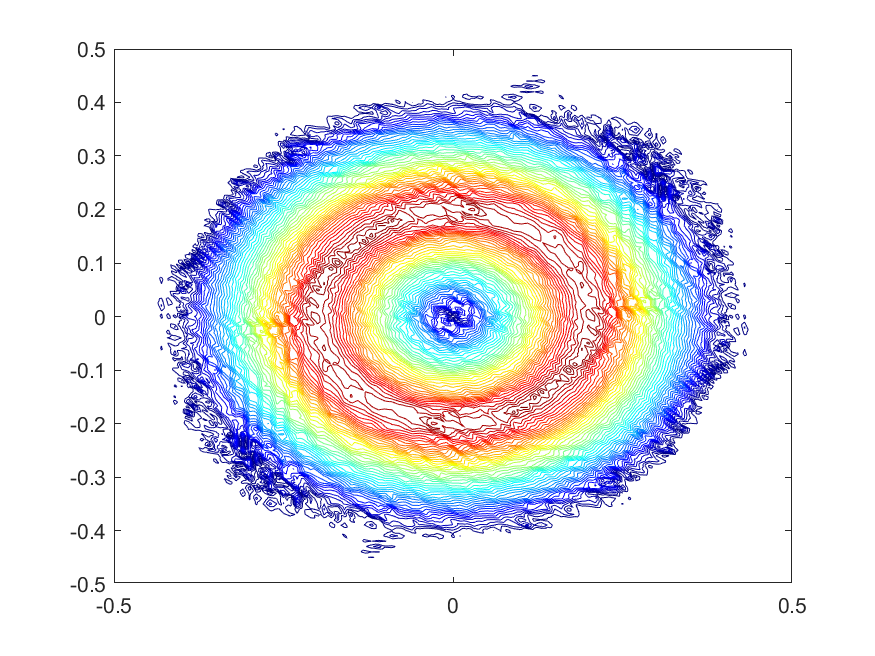}
\includegraphics[scale=.16]{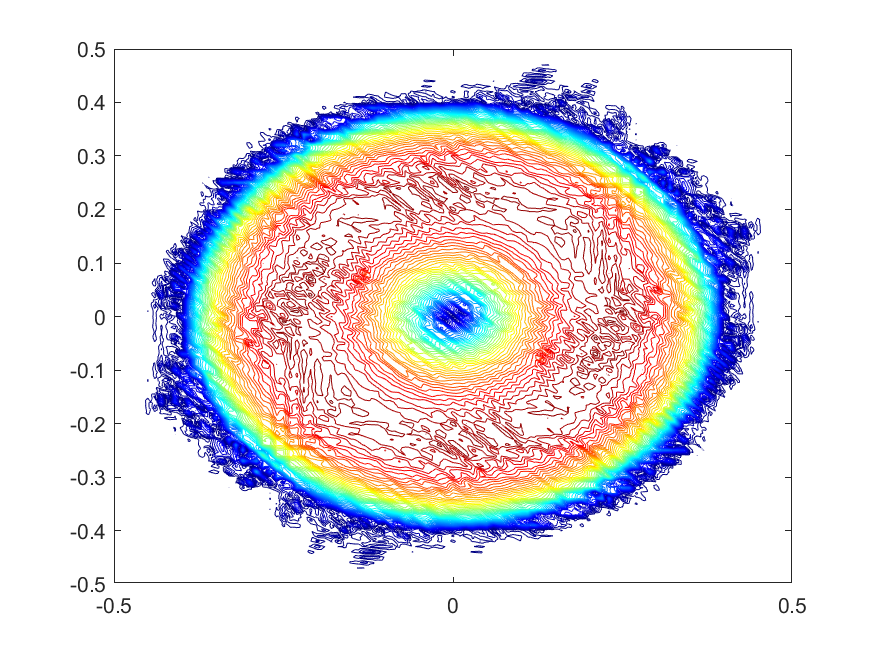}
\includegraphics[scale=.16]{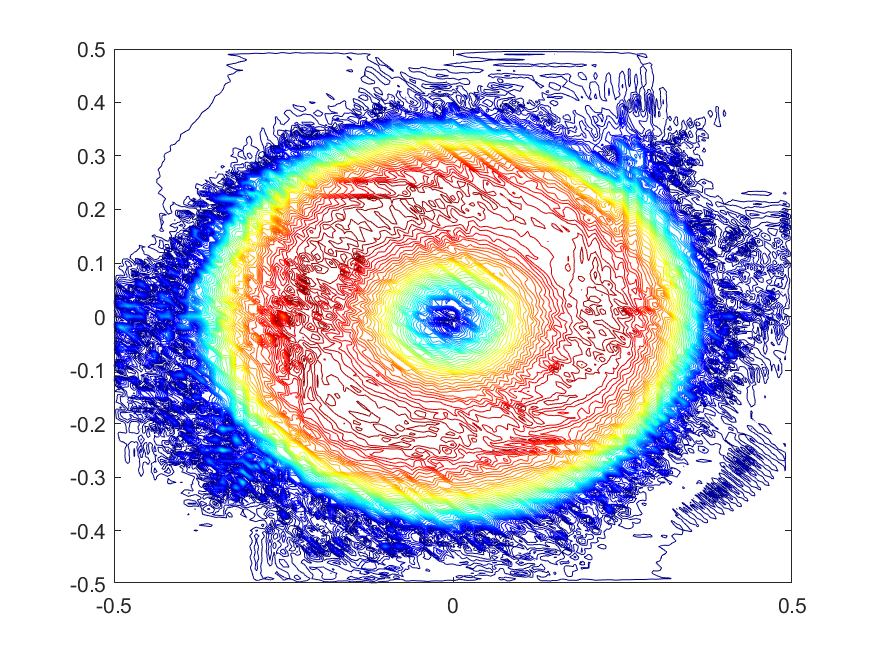}\\
\includegraphics[scale=.16]{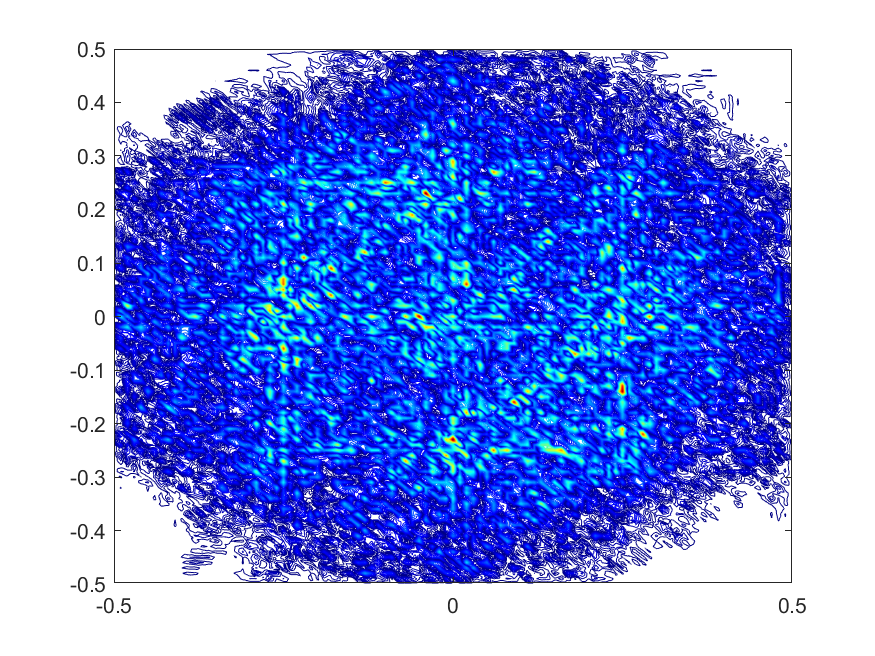}
\includegraphics[scale=.16]{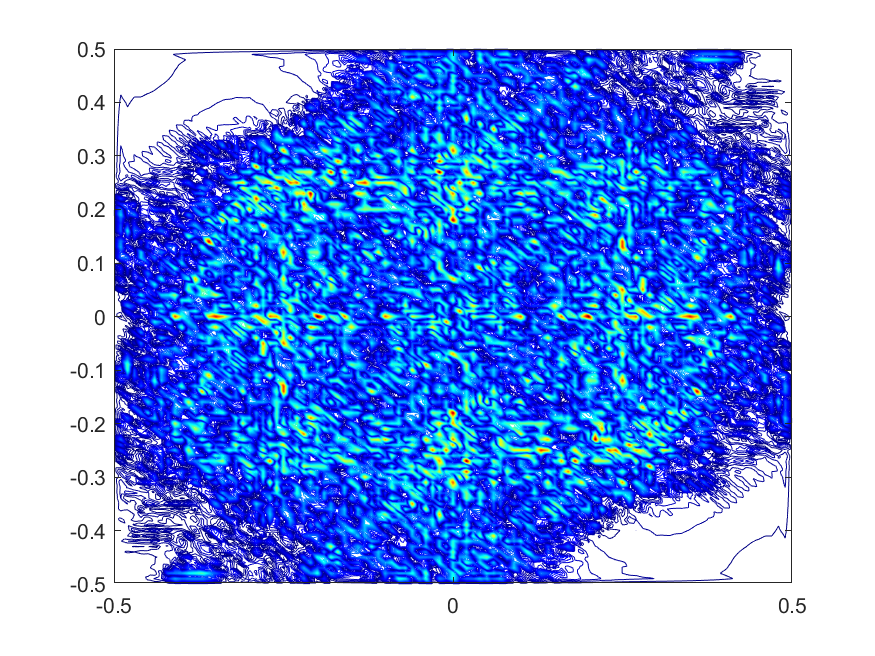}
\includegraphics[scale=.16]{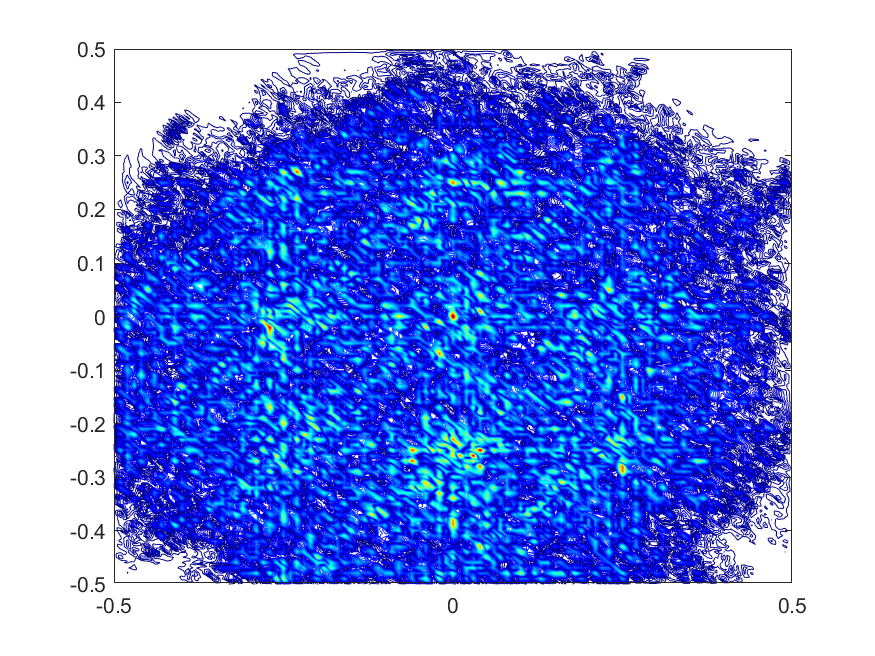}
\includegraphics[scale=.16]{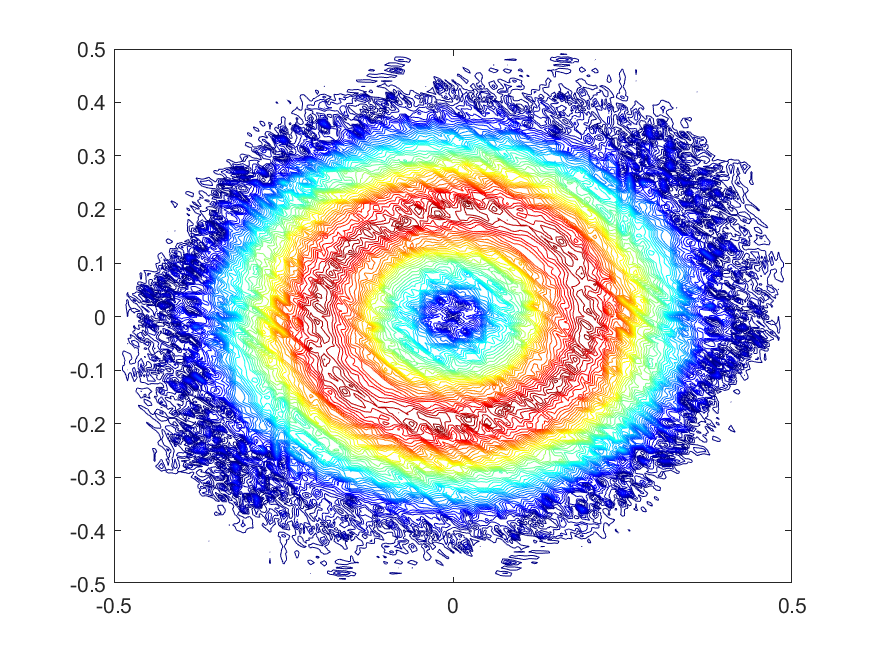}
\includegraphics[scale=.16]{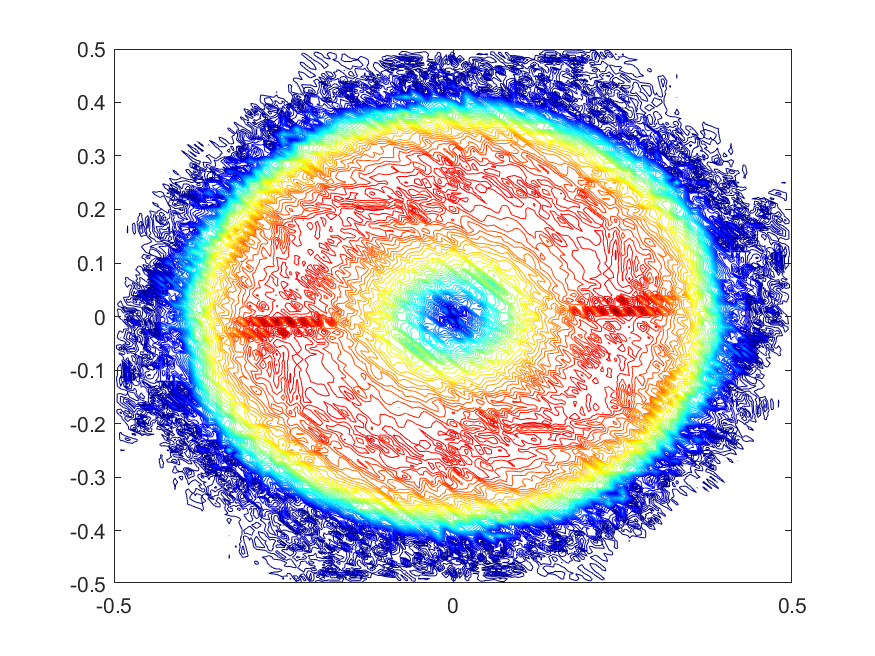}
\includegraphics[scale=.16]{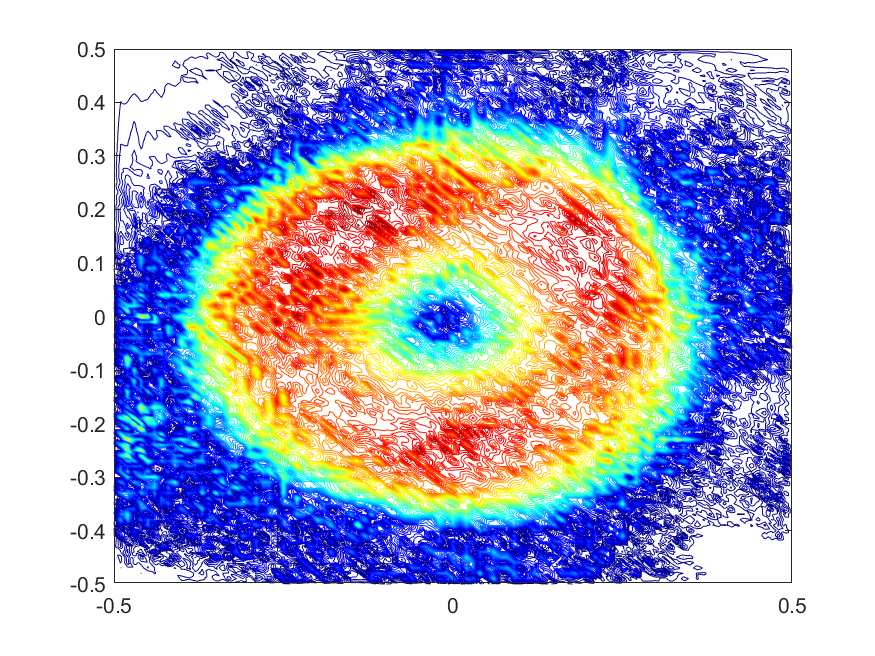}\\
\includegraphics[scale=.16]{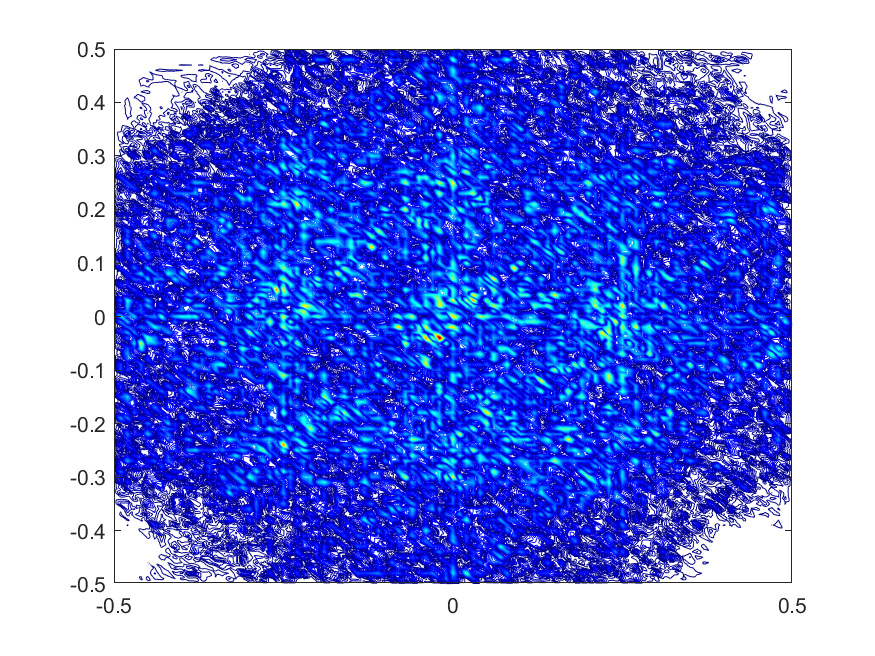}
\includegraphics[scale=.16]{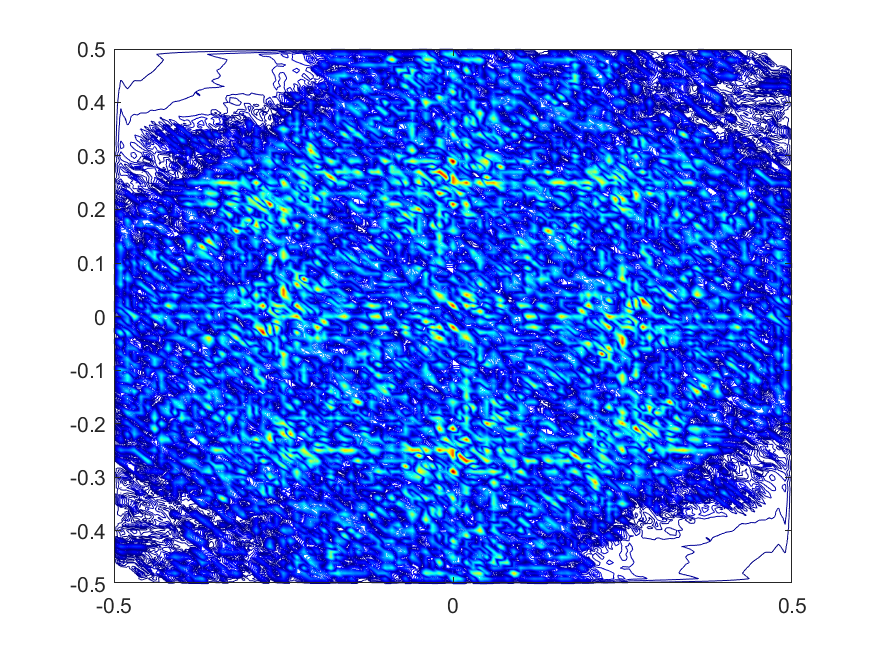}
\includegraphics[scale=.16]{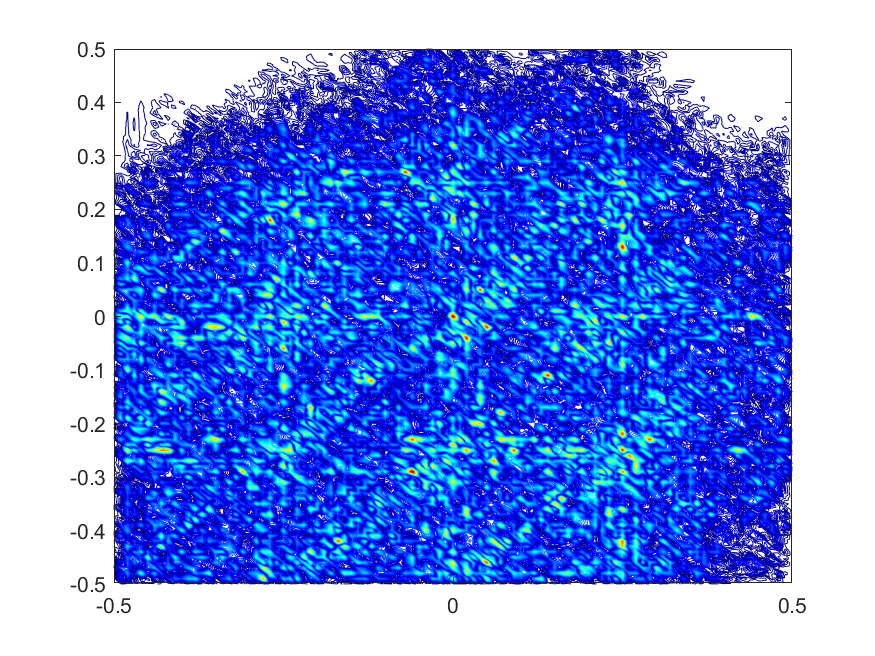}
\includegraphics[scale=.16]{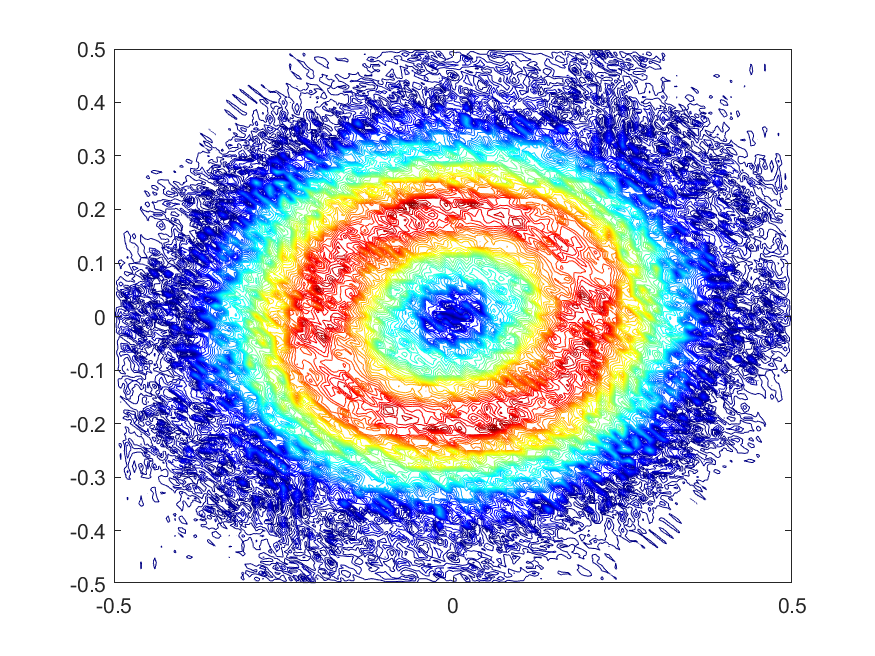}
\includegraphics[scale=.16]{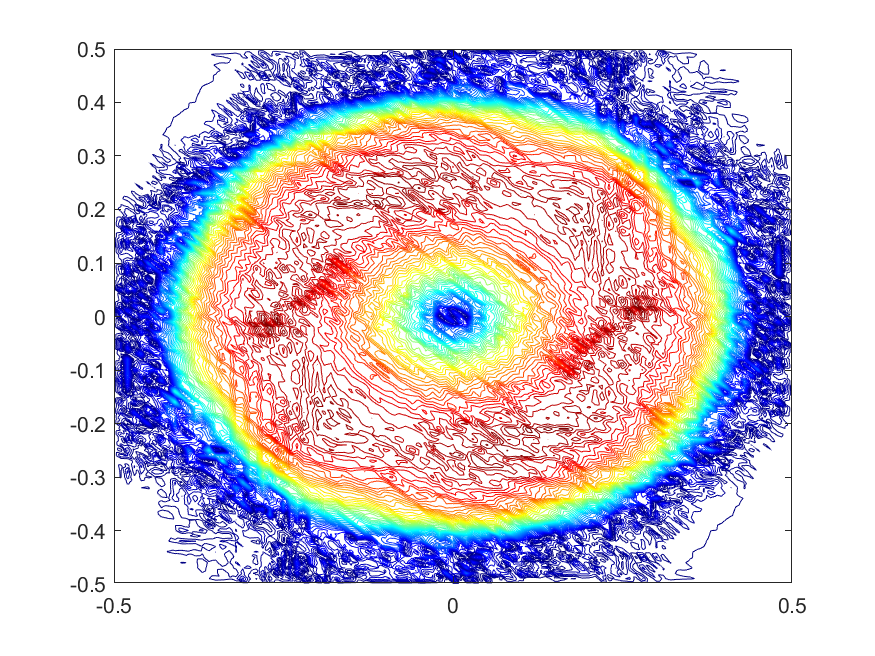}
\includegraphics[scale=.16]{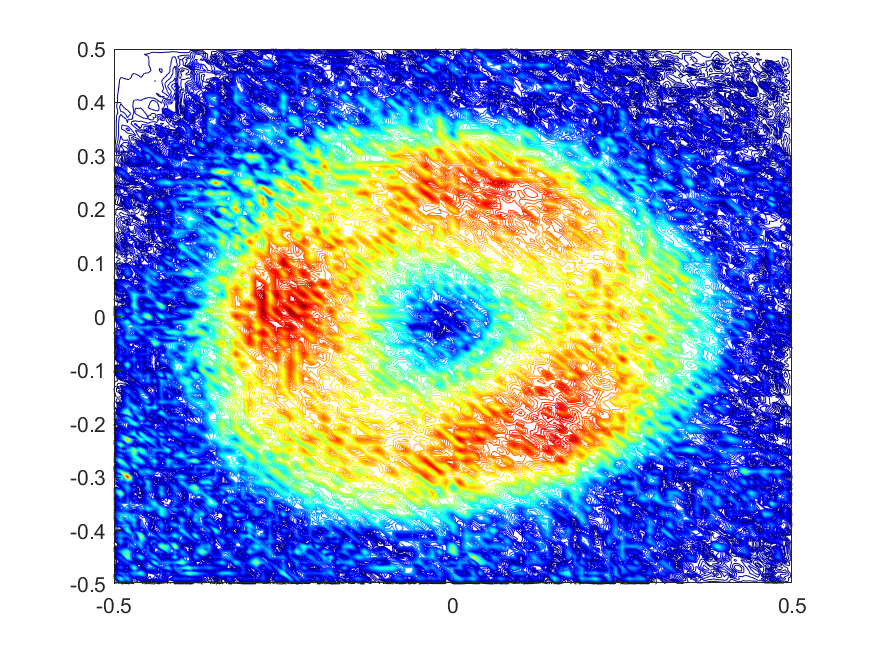}\\
\includegraphics[scale=.16]{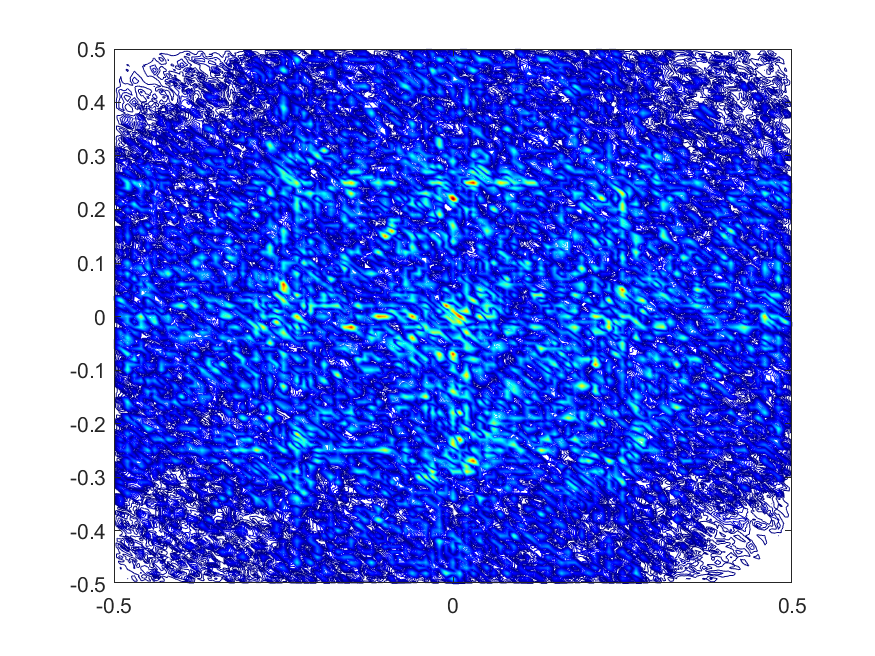}
\includegraphics[scale=.16]{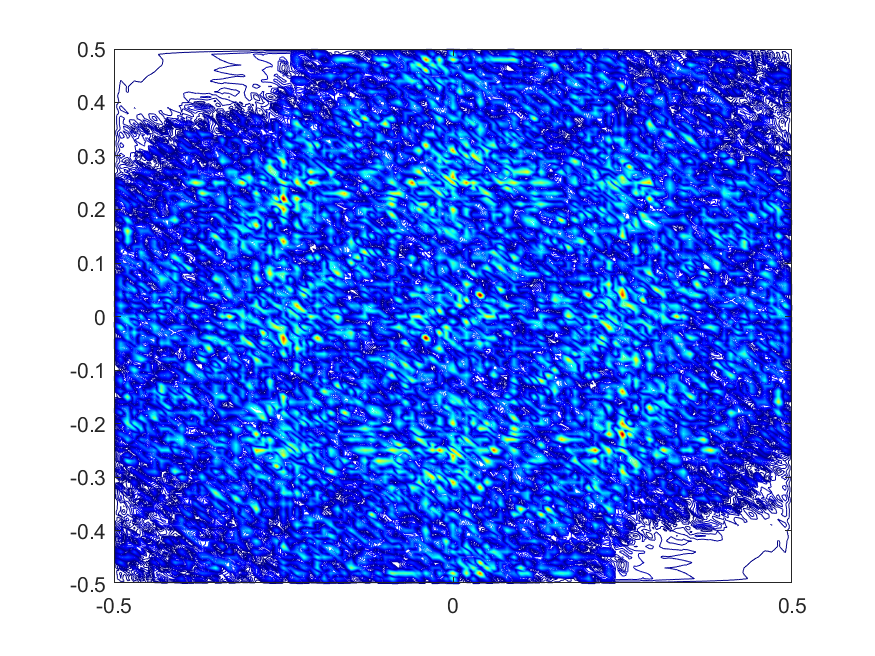}
\includegraphics[scale=.16]{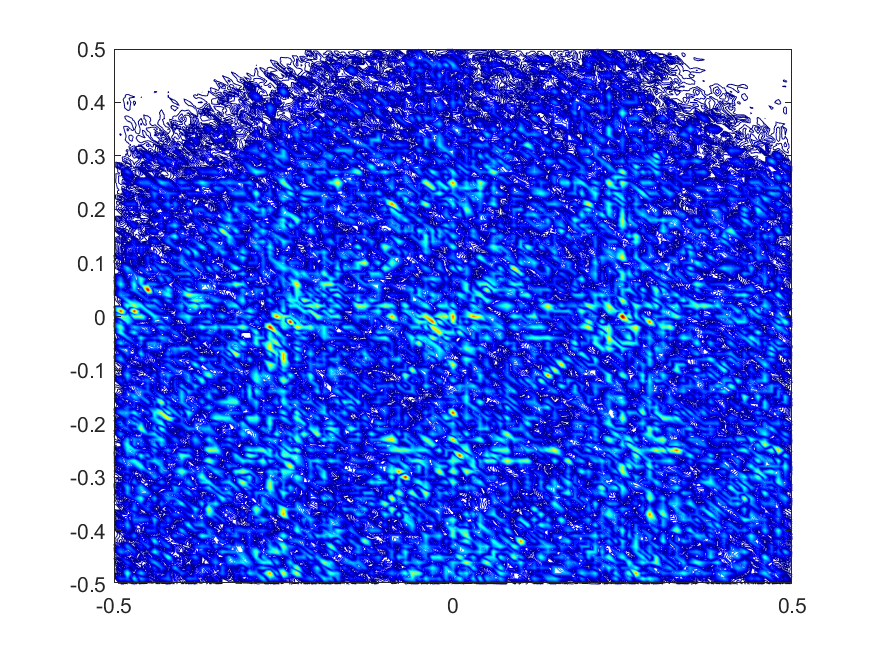}
\includegraphics[scale=.16]{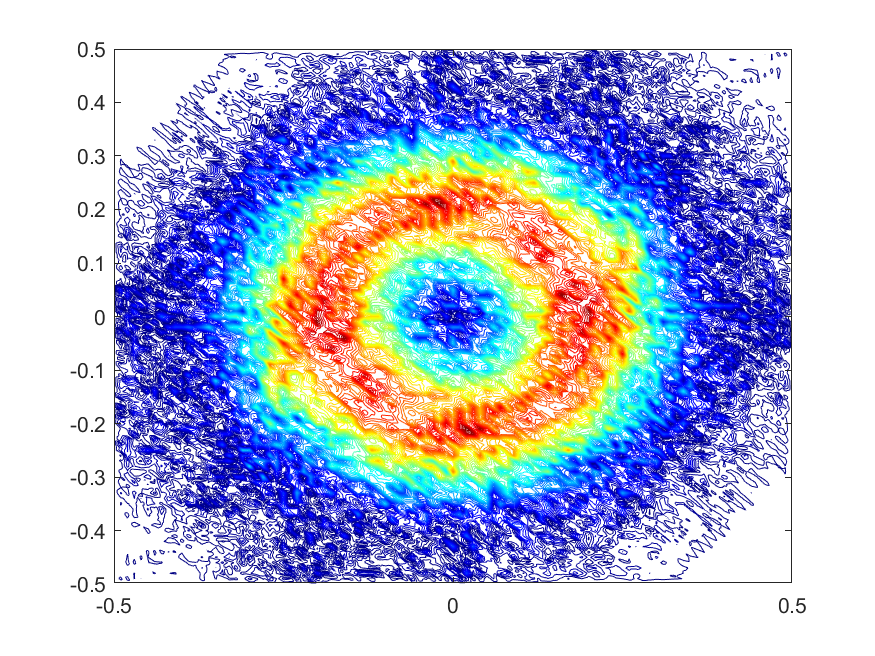}
\includegraphics[scale=.16]{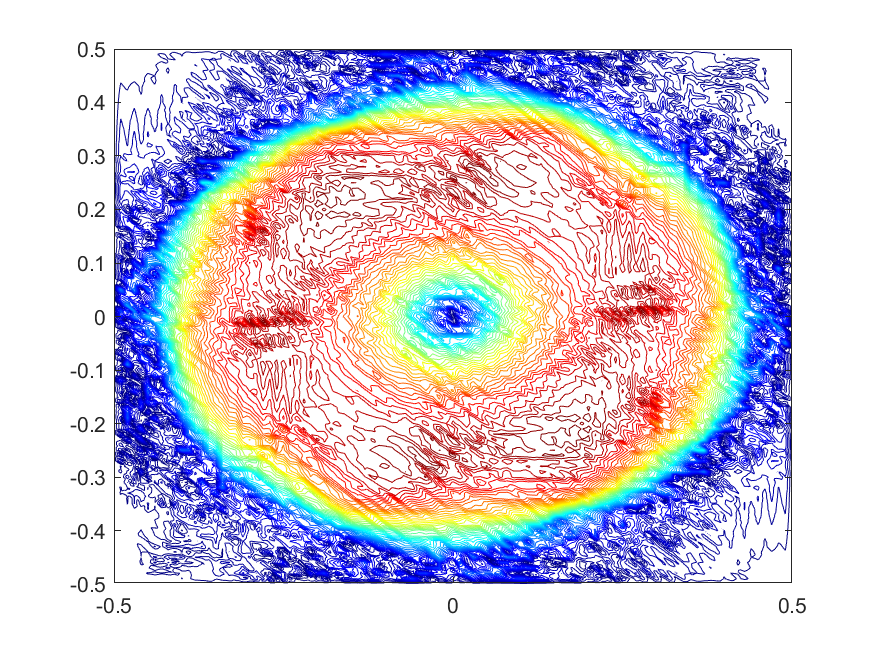}
\includegraphics[scale=.16]{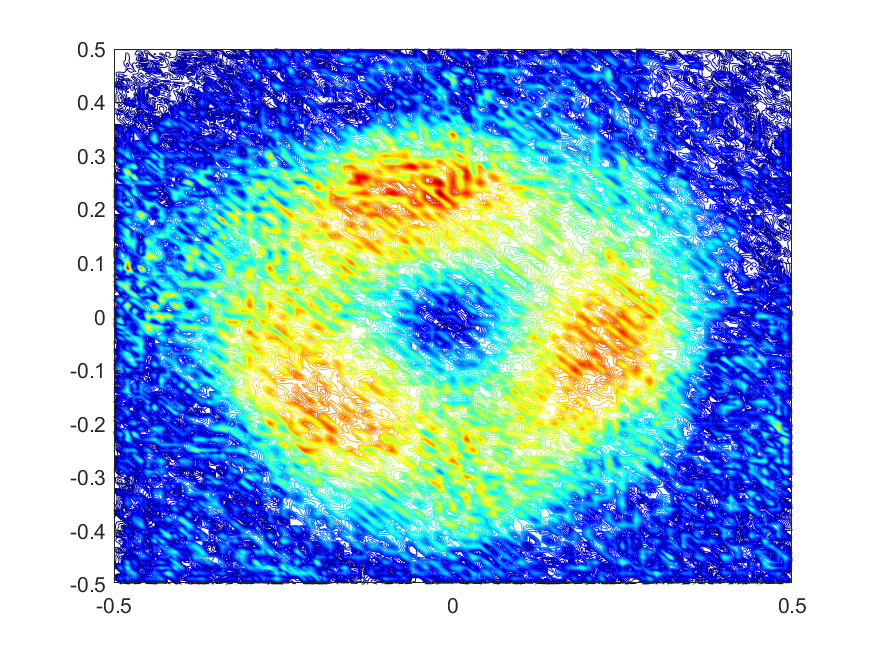}\\
\caption{\textcolor{black}{Numerical results of SKEW coupled, SKEW-BE-PROJ, RotProjB-SKEW, EMAC coupled, EMAC-BE-PROJ, RotProjB-EMAC (from left to right, respectively) at times $t=1,2,3,4$ (top to bottom).}}
\label{fig:Gresho}
\end{figure}

\textcolor{black}{The initial velocity is shown in figure \ref{fig:True_Gresho}, and computed solutions at $t=1,2,3,4$  are shown in figure \ref{fig:Gresho}.
SKEW, SKEW-BE-PROJ, and RotProjB-SKEW produce very poor solutions, as the velocity has dissipated very rapidly compared to the EMAC solutions. Coupled EMAC, EMAC-BE-PROJ, and RotProjB-EMAC do a better (although not great) job at maintaining the vortex for longer times, with the coupled EMAC outperforming the other EMAC methods, especially at later times.}  

\textcolor{black}{Additionally, we calculated $L^2$ error, energy, momentum, and angular momentum for each method.  Coupled EMAC is the only one to conserve energy, with coupled SKEW producing slight dissipation (the severe oscillations in coupled SKEW were enough to cause it to slightly dissipate energy after $t=0.5$).  All projection methods significantly dissipate energy, nonphysically reducing it by nearly 50\% by the end time.  For momentum, the coupled schemes, SKEW-BE-PROJ, and EMAC-BE-PROJ maintain a constant momentum while both projection methods had slight decreases.  We expect this behavior for the four formulations that conserved it, as that is what the theory suggests.  The RotProjB methods both do not have formal conservation analysis done, so we did not specifically expect conservation.}  

\textcolor{black}{For angular momentum, we do not observe conservation for any of the SKEW methods, which is unsurprising.  However we see conservation for both EMAC and EMAC-BE-PROJ, which are of course expected.  If one takes a closer look at the bottom right hand corner of the angular momentum plot, they would notice some small bumps in the angular momentum for RotProjB-EMAC.  Once again, we do not have any conservation results on this formulation, but it does not seem like angular momentum is necessarily conserved here.  Lastly, the $L^2$ error plot shows the EMAC schemes dramatically outperformed the SKEW schemes, with coupled EMAC beating out RotProjB-EMAC and SKEW-BE-PROJ.  Interestingly enough, the $L^2$ error for EMAC-BE-PROJ starts worse off than RotProjB-EMAC, but at around $t=2.5$ it is smaller.  
}

\begin{figure}[H]
\centering
\includegraphics[scale=.4]{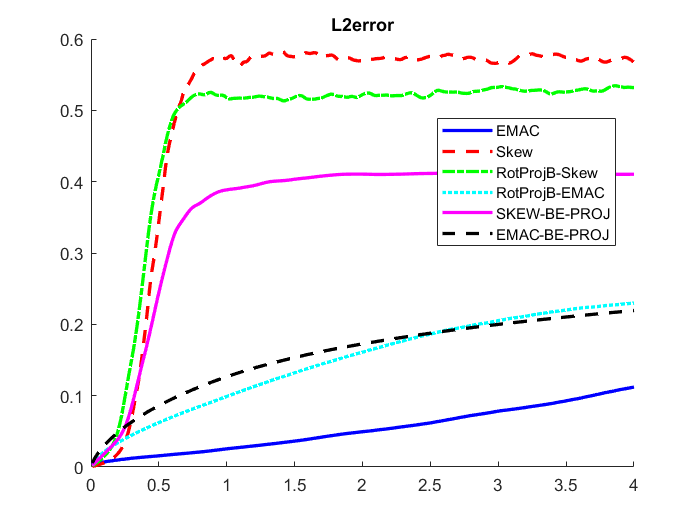}
\includegraphics[scale=.4]{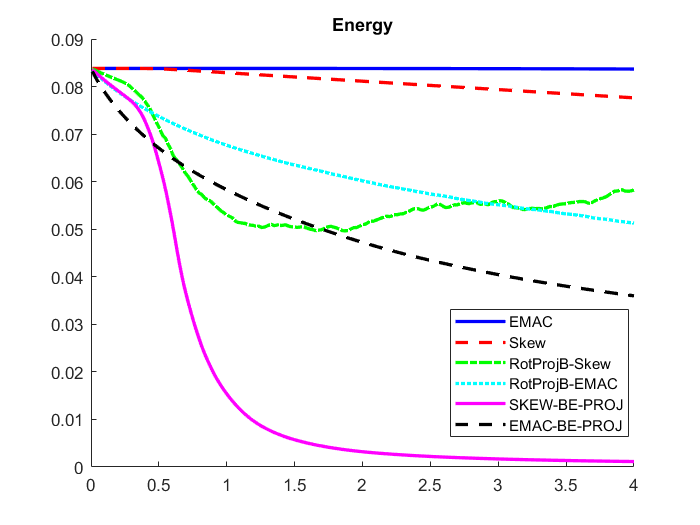}\\
\includegraphics[scale=.4]{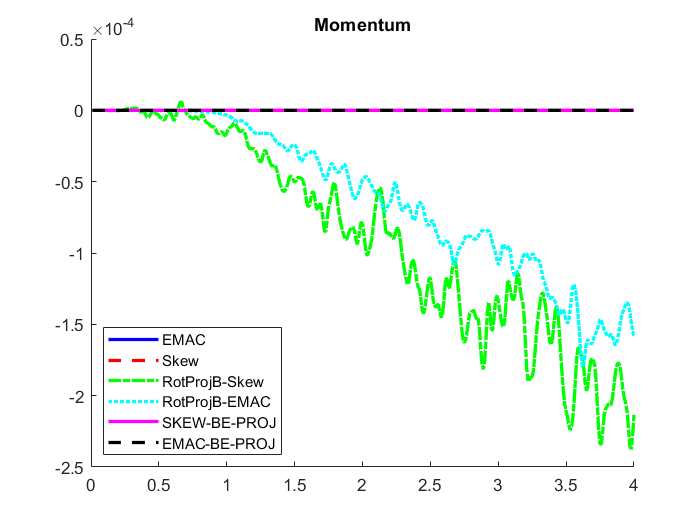}
\includegraphics[scale=.4]{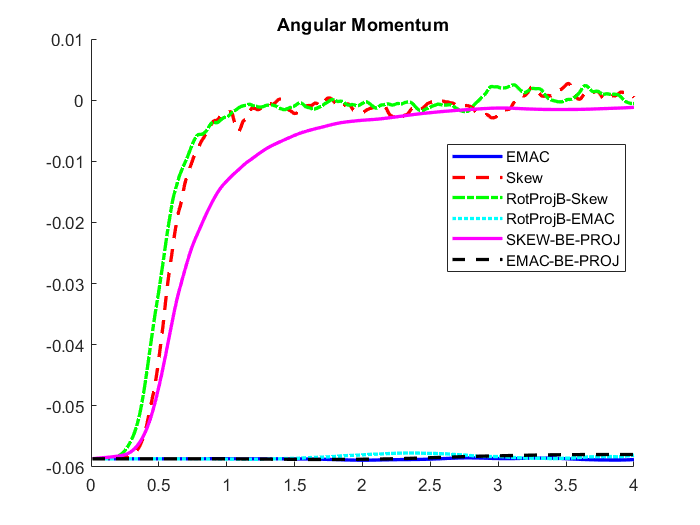}
\caption{\textcolor{black}{$L^2$ error, energy, momentum, and angular momentum plots of SKEW, EMAC, SKEW-BE-PROJ, EMAC-BE-PROJ, RotProjB-SKEW, and RotProjB-EMAC.}}
\label{fig:Gresho_Quant}
\end{figure}

\subsection{Contaminant Flow Analysis}

For our last test, we consider EMAC and SKEW coupled schemes for prediction of river contamination.  We chose the three rivers in Pittsburgh PA (USA) where they meet; in Figure \ref{fig:Pitt_Rivers}, observe the northeastern river (Allegheny River) and the southeastern river (Monongahela River) meet to form the Ohio river. The contaminant is modeled with the fluid transport equation 
\begin{align*}
c_t+u \cdot \nabla c-\eps \Delta c = 0,
\end{align*}
where $c$ is the contaminant, $u$ is the velocity, and \textcolor{black}{$\varepsilon$} is the diffusion coefficient.

\begin{figure}[H]
\centering
\includegraphics[width=.4\textwidth, height=.4\textwidth, viewport=0 0 400 300, clip]{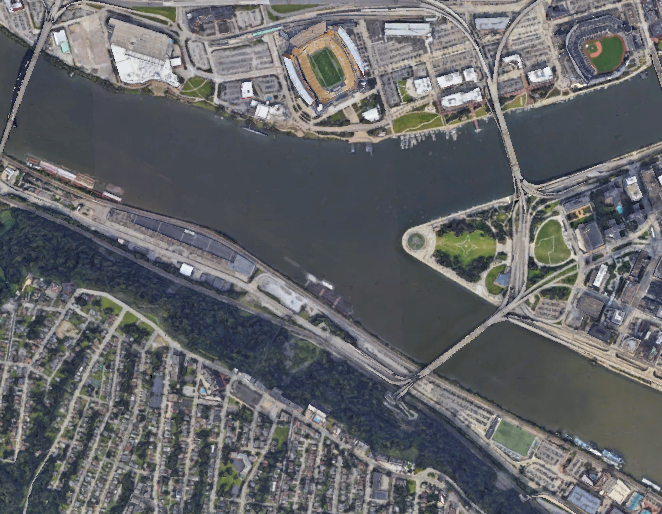}
\includegraphics[width=.45\textwidth, height=.43\textwidth, viewport=0 40 480 390, clip]{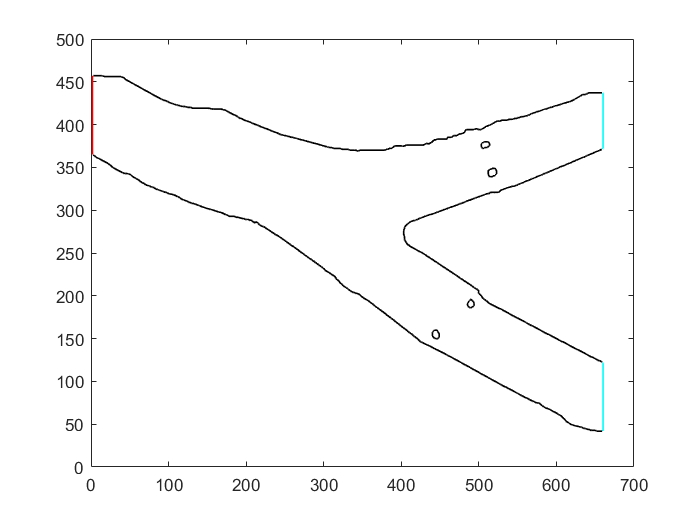}
\caption{Satellite image of the rivers in Pittsburgh, PA}
\label{fig:Pitt_Rivers}
\end{figure}

A domain is created (shown in figure \ref{fig:Pitt_Rivers}) by converting the image to grayscale via the isoline package.  The bridges were pre-edited out, and were replaced with pylons represented by circles.  The domain is such that there are two inlets and one outlet, where the two inlets are the Allegheny and Monongahela Rivers and the outlet is the Ohio River.  The inlets are colored cyan on the right hand side and the outlet is colored red.

We used the Stokes solution for the initial condition, BDF2 timestepping with $\Delta t=.01$ together with Taylor-Hood $(P_2,P_1)$ elements , $Re=\nu^{-1}=100$, and an end time of $T=15$.  Additionally, we used grad-div stabilization with $\gamma=1$ and Newton iterations to solve the nonlinear problem.  For both EMAC and SKEW we used a constant inflow of $u=\bmat{-20,0}^T$ on the Monongahela and Allegheny inflows as well as do-nothing outflow on the Ohio River.  

For the contaminant flow, we used an initial condition of
\begin{align*}
    c=
    \begin{cases}
    1, \quad &\text{if }(x-567)^2+(y-371)^2<25^2,\\
    1, \quad &\text{if }(x-567)^2+(y-131)^2<25^2,\\
    0, \quad &\text{otherwise}.
    \end{cases}
\end{align*}
This gives 2 circles at the same $x$ coordinate on both the Monongahela and Allegheny rivers.  We took $\epsilon=.001$, used $P_2$ elements, BDF2 timestepping and $\Delta t=0.01$.  Lastly, we had do-nothing boundary conditions for each boundary except the inflows, which were set to zero.

We ran simulations with 112229 total degrees of freedom to compare SKEW and EMAC, and also computed a reference solution using 249162 degrees of freedom and the convective nonlinear term, up to time $t=15$.  The solution at $t=15$ is shown in figure \ref{fig:ref_solution}.  EMAC and SKEW solutions are shown in plots at times $t=1, \: 9,$ and 15 in figure \ref{fig:CF_velocity}; we observe that the SKEW solution has significant oscillations which destroy its solution, while the EMAC solution remains stable and exhibits only minor oscillations.  Similarly with the concentrations shown in figure \ref{fig:CF_plots}, where EMAC is stable and matches the resolved solution qualitatively well but SKEW's solution is killed by oscillations.

\begin{figure}[h!]
    \centering
    \includegraphics[scale=.4]{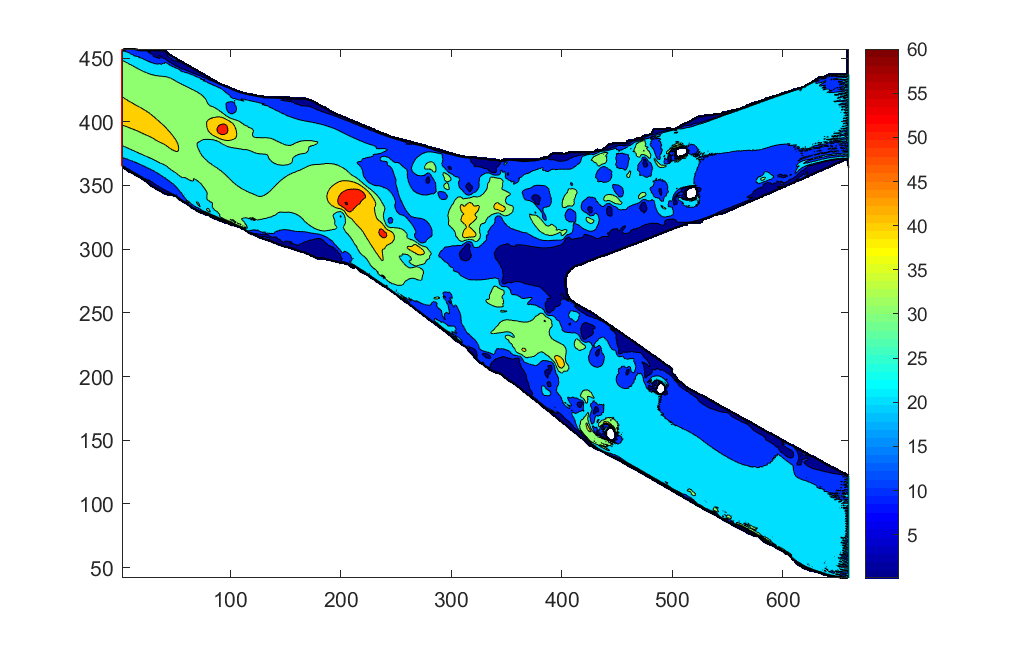}
    \includegraphics[scale=.32]{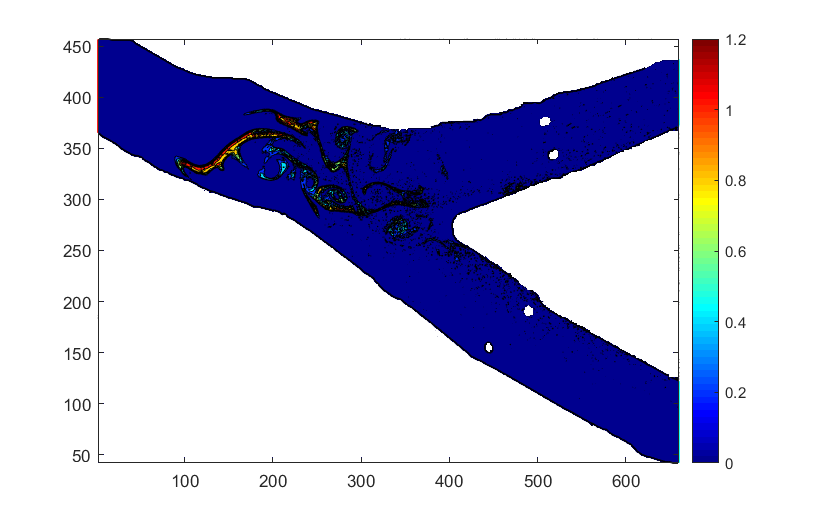}
    \caption{Reference velocity (left) and contaminant flow (right) at time $t=15$.}
    \label{fig:ref_solution}
\end{figure}

\begin{figure}[h!]
    \centering
    \includegraphics[scale=.4]{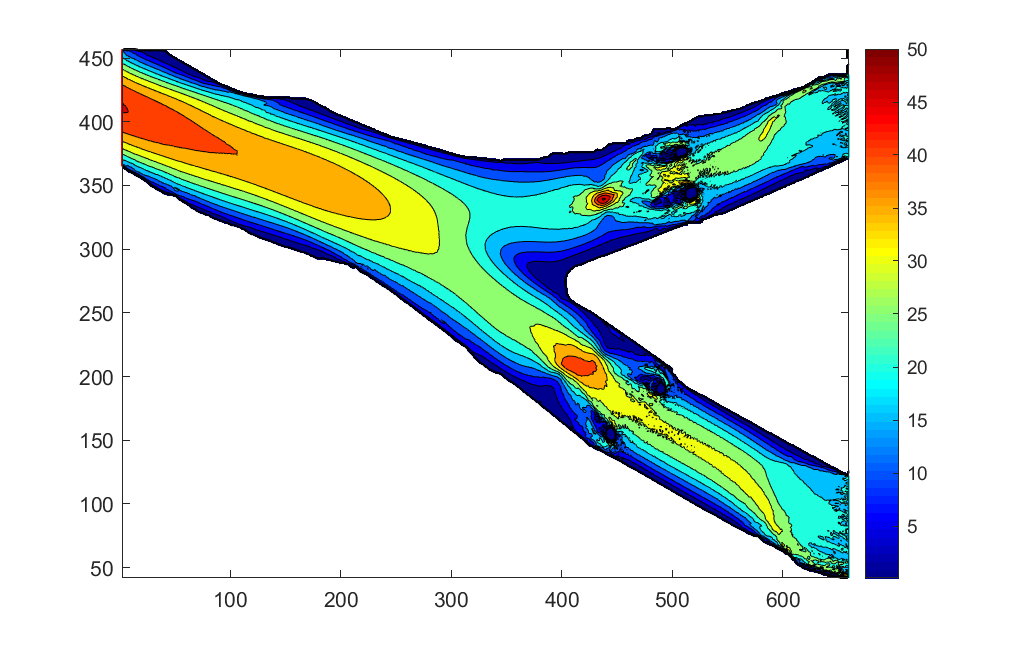}
    \includegraphics[scale=.4]{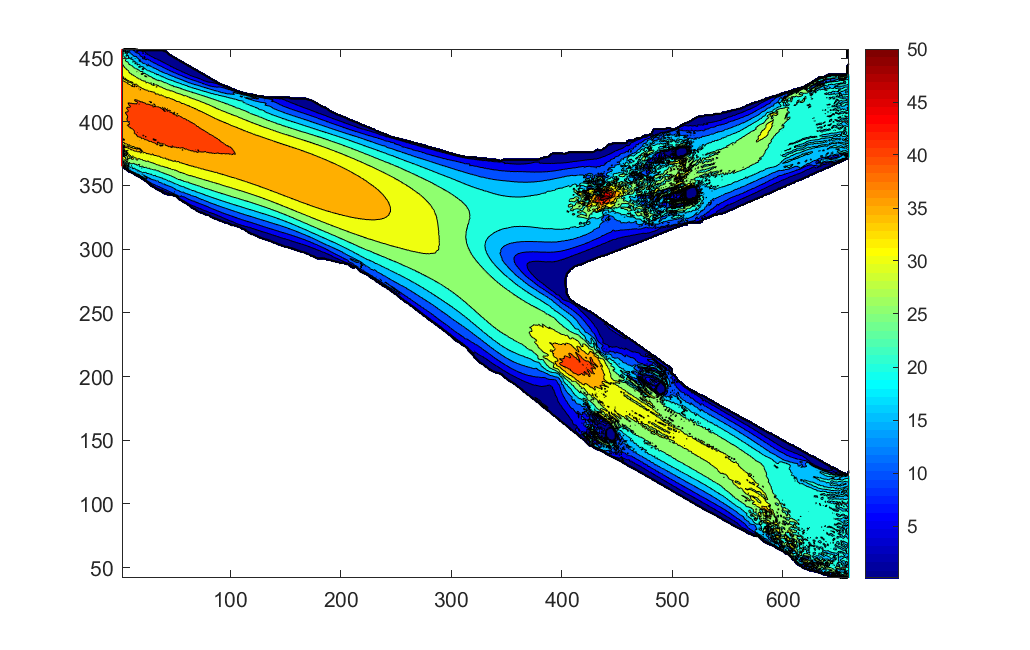}\\
    \includegraphics[scale=.4]{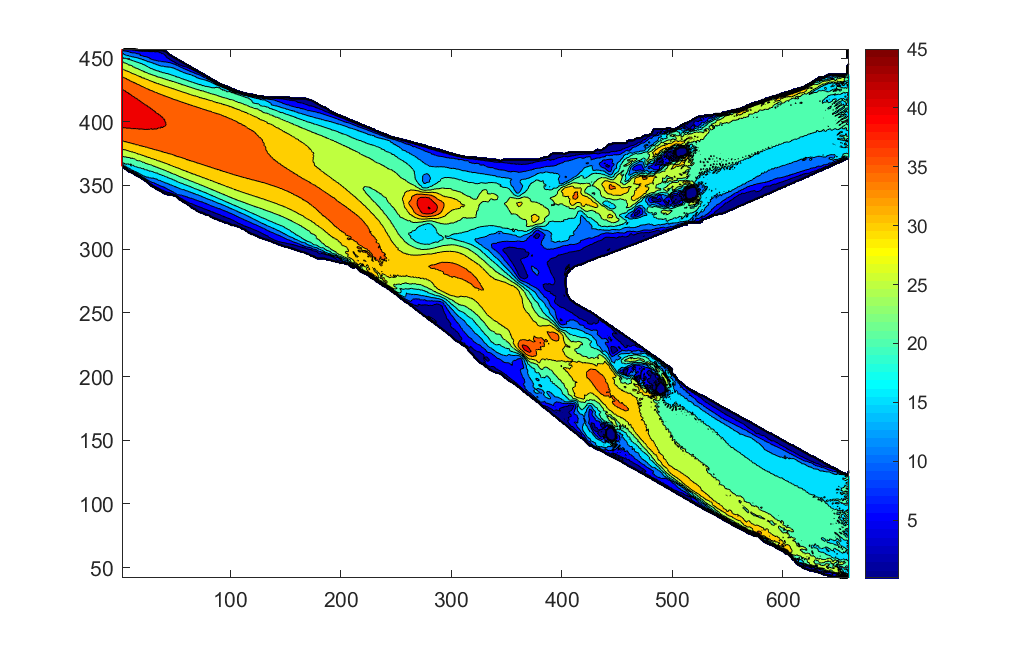}
    \includegraphics[scale=.4]{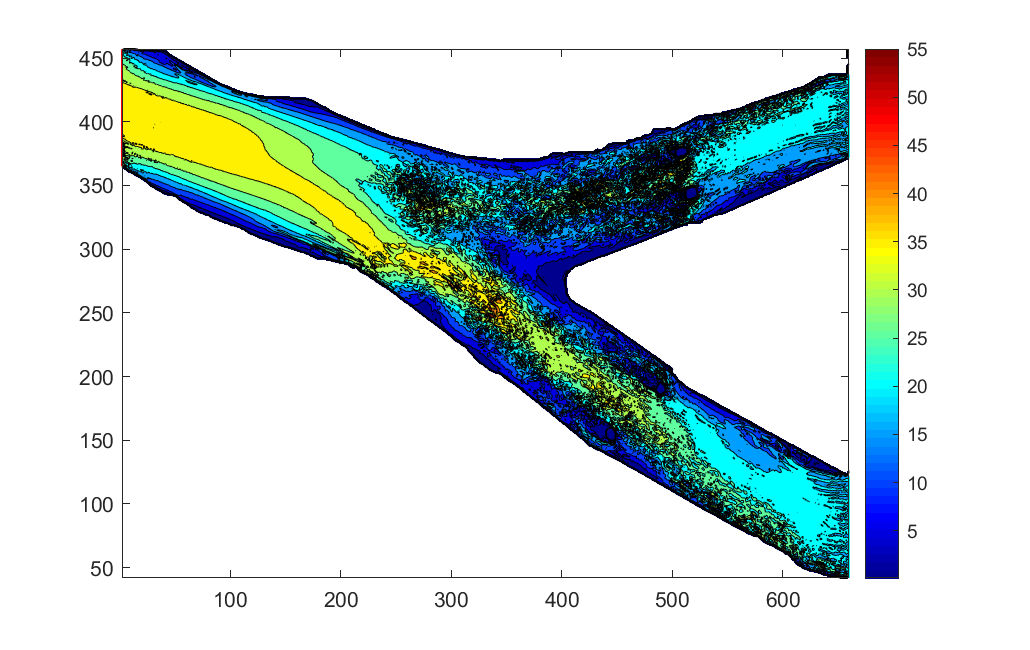}\\
    \includegraphics[scale=.4]{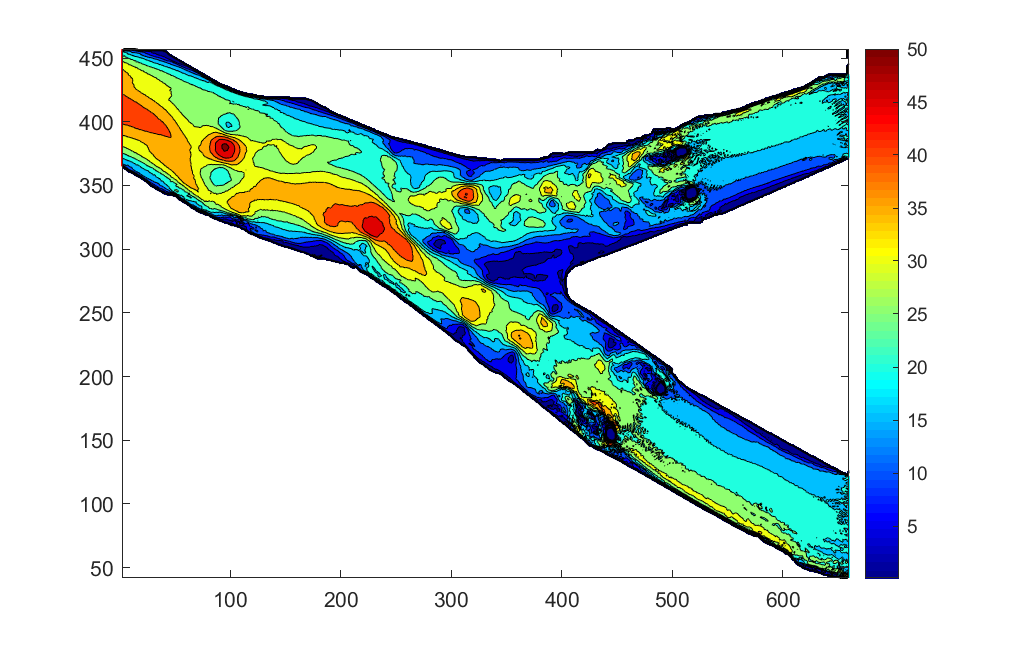}
    \includegraphics[scale=.4]{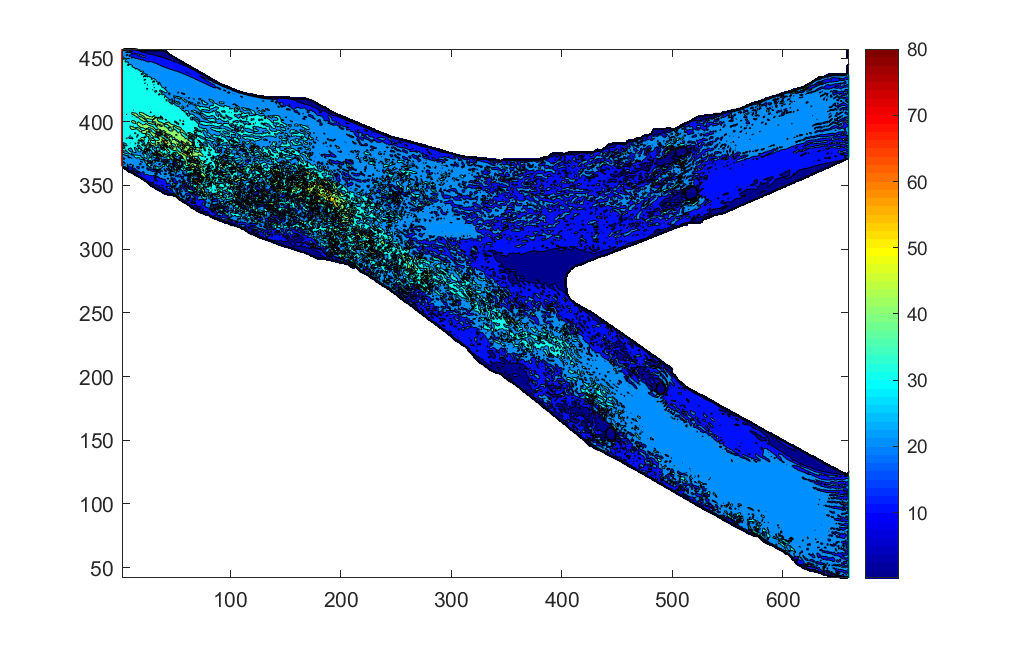}\\
    \caption{Velocity at times $t=3, \: 9,$ and 15 for EMAC (left) versus SKEW (right).}
    \label{fig:CF_velocity}
\end{figure}

\begin{figure}[h!]
    \centering
    \includegraphics[scale=.4]{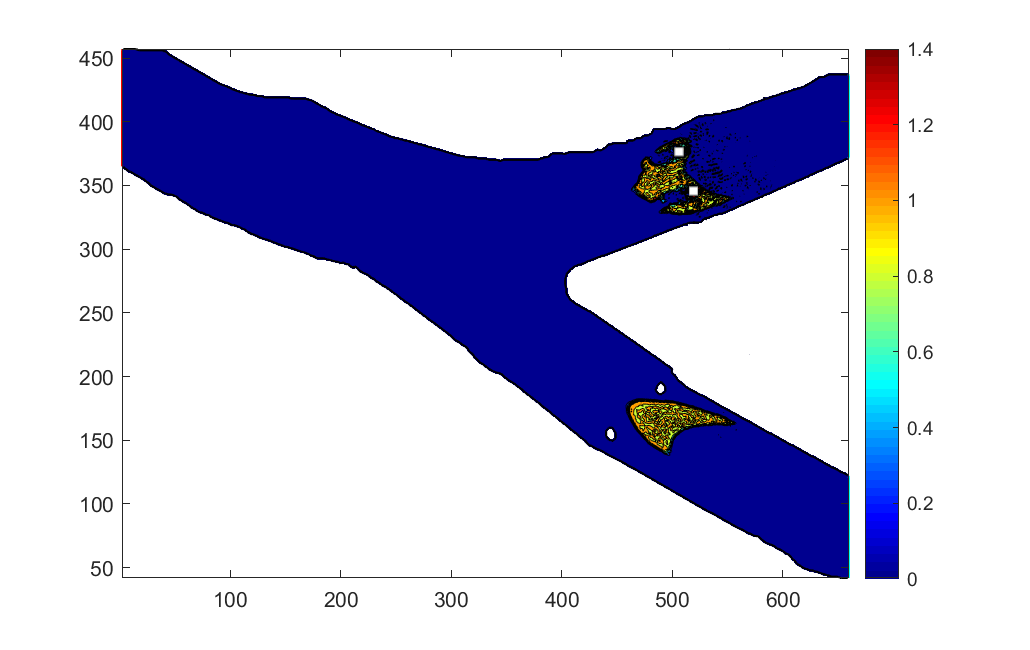}
    \includegraphics[scale=.4]{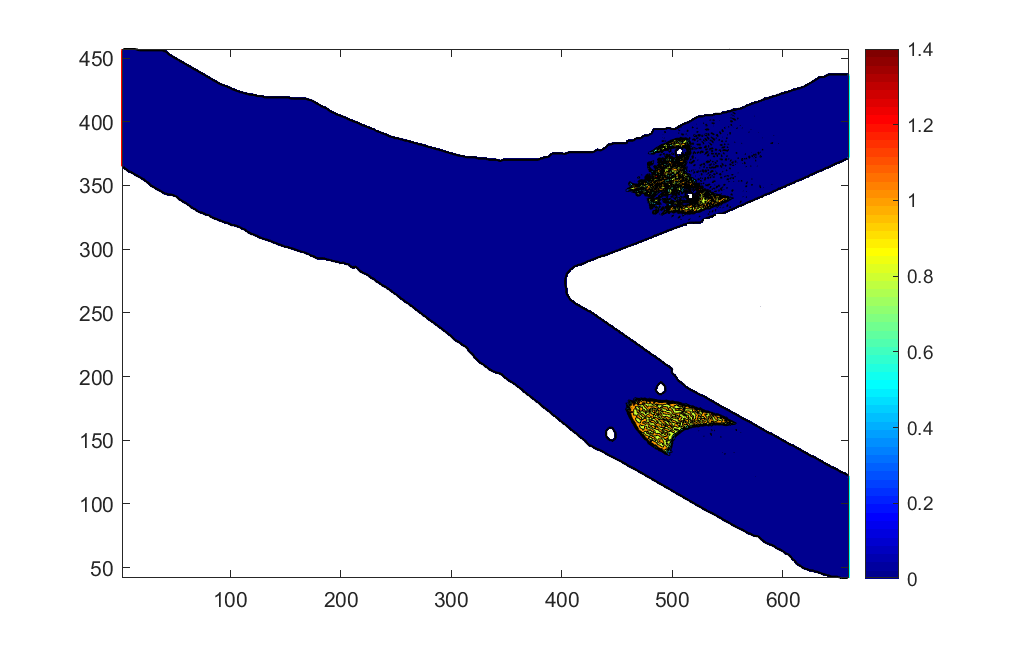}\\
    \includegraphics[scale=.4]{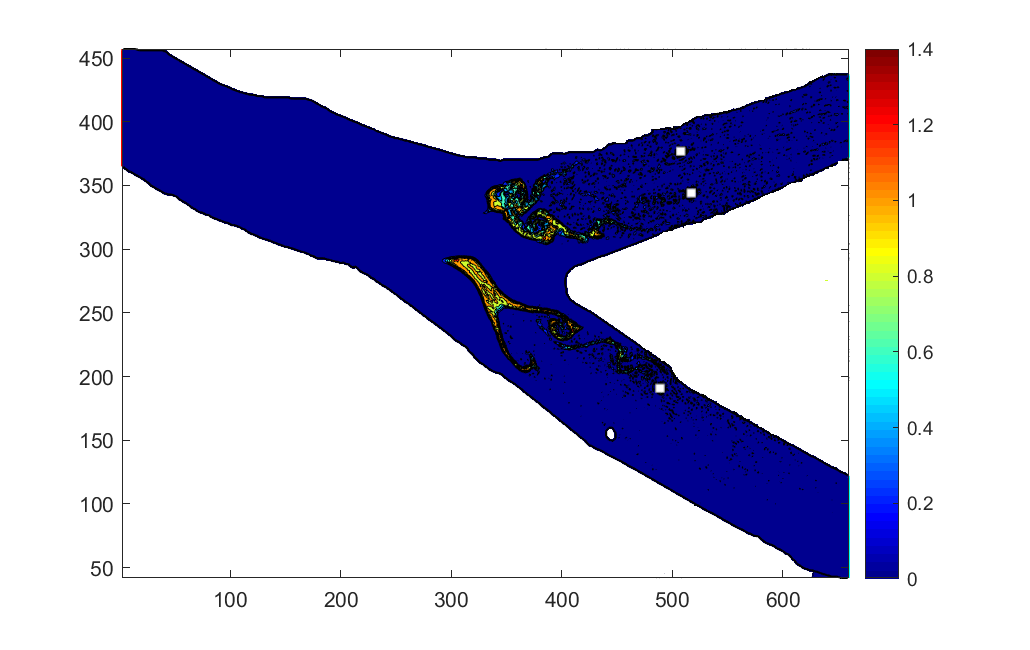}
    \includegraphics[scale=.4]{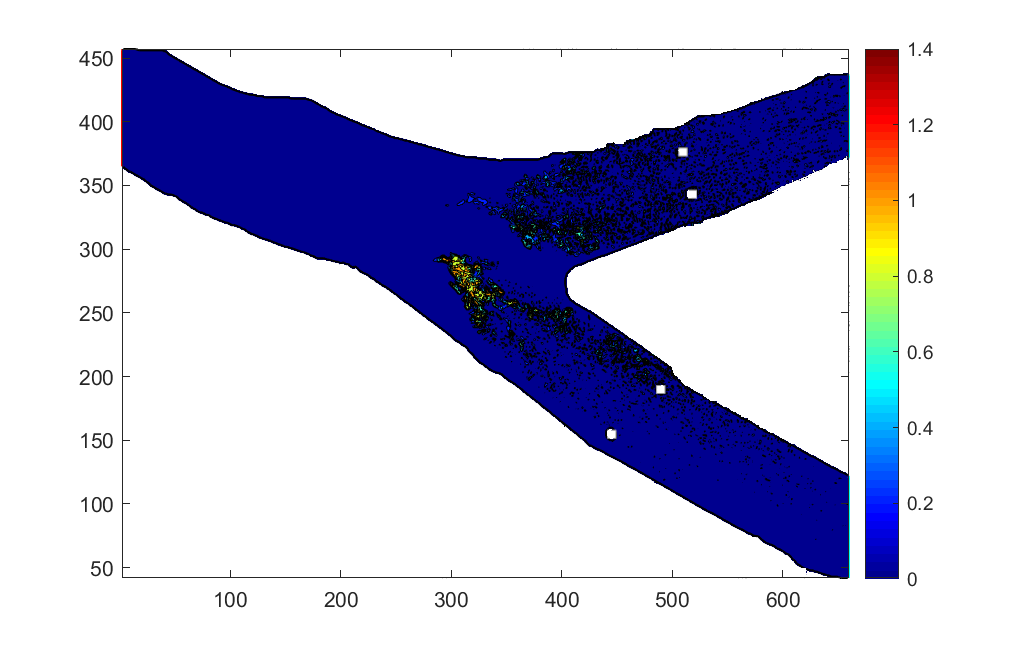}\\
    \includegraphics[scale=.4]{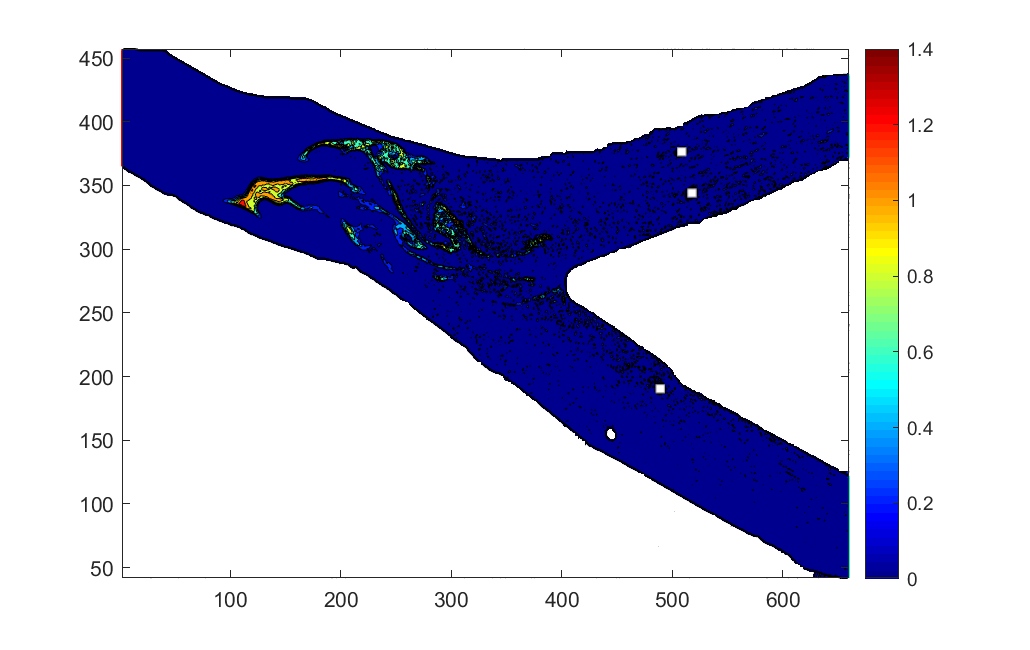}
    \includegraphics[scale=.4]{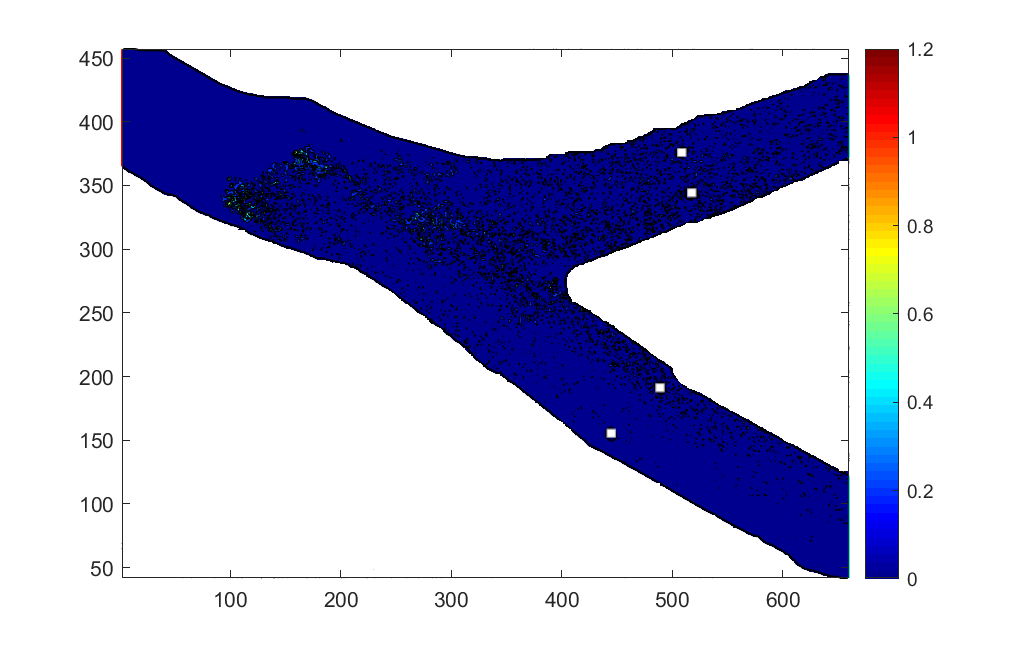}\\
    \caption{Contaminant flow at times $t=3, \: 9,$ and 15 for EMAC (left) versus SKEW (right).}
    \label{fig:CF_plots}
\end{figure}

\section{Conclusions}
\label{sec:conclusion}

We have extended the longer time accuracy analysis of EMAC to both fully discrete projection methods and coupled schemes.  Analysis showed the methods provided better conservation properties than the more commonly used SKEW methods, and that the Gronwall constant from the error bounds for EMAC is significantly reduced compared SKEW in that for EMAC they are not explicitly dependent on the inverse of the viscosity.  Several numerical tests backed up the analysis, and agreed with what is now found in many computational works since EMAC first appeared in the literature in 2017: EMAC performs better than the analogous (i.e. coupled, projection, etc.) method using SKEW, especially in longer time simulations.

\bibliographystyle{plain}
\bibliography{bib_file}

\end{document}